\documentclass[12pt]{amsart}

\usepackage{graphicx} % Required for inserting images
\usepackage[utf8]{inputenc}
\usepackage[margin=1in]{geometry}
\usepackage{amsmath}
\usepackage{amsfonts}
\usepackage{amsthm}
\usepackage{amssymb}
\usepackage{array}
\usepackage{cancel,xcolor}
\usepackage{multirow}
\usepackage{graphicx}
\usepackage{tikz}
\usepackage{comment}
\usepackage{dynkin-diagrams}
\usetikzlibrary{positioning}
\usepackage{array}
\usepackage{float}

\newcolumntype{C}[1]{>{\centering\arraybackslash}m{#1}}
\newcommand{\centertikz}[1]{%
\begin{minipage}{\linewidth}
\centering
#1
\end{minipage}
}

\newtheorem*{theorem*}{Theorem}

\newtheorem{theorem}{Theorem}[section]
\newtheorem{example}[theorem]{Example}

\newtheorem{coro}[theorem]{Corollary}
\newtheorem{lemma}[theorem]{Lemma}
\newtheorem{definition}[theorem]{Definition}
\newtheorem{proposition}[theorem]{Proposition}

\newtheorem{remark}[theorem]{Remark}
\newcommand{\Z}{\mathbb{Z}}

\usepackage{tabularray}

\usepackage{biblatex}
\begin{document}

\title{The Poincar\'e series of Coxeter folding subgroups}
\author{Camilo Villamil}
\address{Department of Mathematics, Oklahoma State University, 401 Mathematical Sciences, Stillwater, OK 74078}
\email{cvillam@okstate.edu}

\author{Edward Richmond}
\address{Department of Mathematics, Oklahoma State University, 401 Mathematical Sciences, Stillwater, OK 74078}
\email{edward.richmond@okstate.edu}

\maketitle

\begin{abstract}
    Folding subgroups give a way to realize non-simply-laced Coxeter groups as subgroups of simply-laced Coxeter groups.  In this paper, we study how folding subgroups of finite and affine type are distributed length-wise by calculating the length generating function of the subgroup with respect the length of the ambient group.  These generating functions have surprisingly nice formulas in terms of $q$-integers and give rise to interesting combinatorial identities on polynomials involving length statistics of both the ambient group and folding subgroup.
\end{abstract}

\section{Introduction}

Let $(W,S)$ be a Coxeter system and let $\ell:W\rightarrow \Z_{\geq 0}$ denote its length function.  Let $\widehat S$ be a set partition of $S$ such that for each $I\in\widehat S$, the parabolic subgroup $W_I=\langle s\ |\ s\in I \rangle$ is a finite subgroup.  Let $r_I$ denote the maximal length element in $W_I$ and define $\widehat W$ as the subgroup of $W$ generated by $R:=\{r_I \ |\ I\in \tilde S\}$. 
 We say the set partition $\widehat S$ is \textbf{admissible} if for every $w\in \widehat W$ and $I\in \widehat S$, we have either $\ell(ws)=\ell(w)+1$ for each $s\in I$ or  $\ell(ws)=\ell(w)-1$ for each $s\in I$.  The following theorem is proved by  M\"uhlherr in \cite{muhlherr1993coxeter}.
 
\begin{theorem}[M\"uhlherr, 1992]\label{thm:folding_subgroup}
    Let $(W,S)$ be as above and let $\widehat{S}$ be an admissible partition of $S$. Then the pair $(\widehat{W}, R)$ is a Coxeter system.    
\end{theorem}

We call the Coxeter subgroup $\widehat W$ from  Theorem \ref{thm:folding_subgroup}\textbf{ a folding subgroup of $W$}.  Folding subgroups appear naturally in the study of simple Lie groups and their associated Weyl groups.  For example, the natural embedding of the symplectic group into the special linear group $Sp(2n)\subset SL(2n)$ induces a folded embedding of the Weyl group of $Sp(2n)$ into the Weyl group of $SL(2n)$. 
 The term ``folding" comes from the way the set partition $\widehat S$ collects vertices in the Coxeter graph of $W$ to obtain the Coxeter graph of $\widehat W$.  The folding of Coxeter groups is a natural consequence of the folding of root systems which were described by Steinberg in \cite{steinberg1967lectures} for finite crystallographic root systems and later by Lusztig for general finite root systems in \cite{lusztig1983some}.

 %(see Table \ref{DescriptionFiniteFoldingTable} for finite Coxeter foldings, and Table \ref{DescriptionAffineFoldingTable1} and Table \ref{DescriptionAffineFoldingTable2} for affine Coxeter foldings).\\

The folding process establishes a relation between the canonical generators of both $\widehat W$ and $W$. Since any element $w\in\widehat{W}$ can be expressed as a word in $R-$generators and in $S-$generators, we distinguish these expressions by considering the embedding map
$$\phi:\widehat W\rightarrow W.$$  
If $w$ is an expression in the generators in $R$, then $\phi(w)$ will denote the corresponding expression in the generators in $S$.  We call $\phi(w)$ the \textbf{unfolding} of $w$. We study how folding subgroups are distributed by length in their ambient group.  To measure this distribution, we make the following definition.

\begin{definition}
Let $W$ be a Coxeter group with folding subgroup $\widehat{W}$.  Define the \textbf{unfolding series}
    \begin{equation*}
        U_{\widehat W}^{W}(q):=\sum_{w\in \widehat W}q^{\ell(\phi(w))}.
    \end{equation*}
If $\widehat W$ is finite, then $U_{\widehat W}^{W}(q)$ is the \textbf{unfolding polynomial}.
\end{definition}

We focus on folding subgroups of finite and affine type and discuss each type separately in the following two subsections.

%In this paper, we study foldings of 
%Coxeter groups of both finite and affine type.

\subsection{Results for foldings of Coxeter groups of finite type}

%Irreducible Coxeter groups of finite type are %classified into four infinite families and seven %exceptional types.  

For irreducible Coxeter groups of finite type, folding pairs $(\widehat W, W)$ are completely classified  by Crisp in \cite{crisp1996injective}.  We list these foldings in Table \ref{FiniteFoldingTable}.  Here we use the notational conventions for the classification of finite Coxeter groups found in \cite[Appendix A1]{bjorner2005combinatorics}. 

%The folding pairs $(\tilde W, W)$ in the affine case %are listed in Table \ref{AffineFoldingTable}.\\

\begin{table}[ht]
    \centering
    \begin{tabular}{ccc}    
         & $\qquad\qquad\widehat{W}\qquad\qquad$ & $\qquad W\qquad$ \\
         \hline
        (i) & ${B}_n$ & $A_{2n-1}$, $A_{2n}$, $D_{n+1}$\\
        (ii) & $I_2(n)$ & $A_{n-1}$\\
        (iii) & $F_4$ & $E_6$\\
        (iv) & $H_3$ & $D_6$\\
        (v) & $H_4$ & $E_8$\\
        (vi) & $G_2$ & $D_4$\\
    \end{tabular}
    \vspace{10pt}
    \caption{Folding subgroups of finite Coxeter groups}
    \label{FiniteFoldingTable}
\end{table}

We study the unfolding polynomials in the four infinite families given in parts (i) and (ii) in Table \ref{FiniteFoldingTable} (note that parts (iii)-(vi) are not infinite families and involve Coxeter groups of exceptional Lie-type).  The Coxeter graphs of these families are given in Table \ref{tab:groupstable}.

\begin{table}[h!] 
   \centering 
   \begin{tabular}{
>{\centering\arraybackslash}m{0.03\linewidth}
>{\centering\arraybackslash}m{0.35\linewidth}
>{\centering\arraybackslash}m{0.03\linewidth}
>{\centering\arraybackslash}m{0.28\linewidth}
} 
    
   $A_n:$  \vspace{-15pt} & \centertikz{\begin{tikzpicture}[
      thick,
      acteur/.style={
        circle,
        fill=black,
        thick,
        inner sep=2pt,
        minimum size=0.2cm
      },
      baseline={(0,0.5)}
    ] 
      
\node (c1) at (8,2) [acteur,label=above:$s_{n-1}$]{};
      \node (c2) at (7,2)[acteur]{}; 
      \node (c3) at (6,2) [acteur,label=above:$s_2$]{}; 
      \node (c4) at (5,2) [acteur,label=above:$s_1$]{}; 
      \node (c5) at (9,2) [acteur,label=above:$s_n$]{}; 
%      \node (a6) at (2.5,0.5) [acteur,label=below:$s_{n+1}$]{};
%      \node (a7) at (1,0.5) [acteur]{}; 
%      \node (a8) at (-0.5,0.5) [acteur,label=below:$s_{2n-2}$]{}; 
%      \node (a9) at (-2,0.5) [acteur,label=below:$s_{2n-1}$]{}; 
      \draw (c1) -- (c2); 
      \draw[dashed] (c2) -- (c3); 
      \draw (c3) -- (c4);
      \draw (c1) -- (c5);
    \end{tikzpicture}} \vspace{-30pt}
    &
    $B_n:$  \vspace{-15pt}& \centertikz{\begin{tikzpicture}[
      thick,
      acteur/.style={
        circle,
        fill=black,
        thick,
        inner sep=2pt,
        minimum size=0.2cm
      },
      baseline={(0,0.5)}
    ] 
      
\node (c1) at (8,2) [acteur,label=above:$s_{n-1}$]{};
      \node (c2) at (7,2)[acteur]{}; 
      \node (c3) at (6,2) [acteur,label=above:$s_2$]{}; 
      \node (c4) at (5,2) [acteur,label=above:$s_1$]{}; 
      \node (c5) at (9,2) [acteur,label=above:$s_n$]{}; 
%      \node (a6) at (2.5,0.5) [acteur,label=below:$s_{n+1}$]{};
%      \node (a7) at (1,0.5) [acteur]{}; 
%      \node (a8) at (-0.5,0.5) [acteur,label=below:$s_{2n-2}$]{}; 
%      \node (a9) at (-2,0.5) [acteur,label=below:$s_{2n-1}$]{}; 
      \draw (c1) -- (c2); 
      \draw[dashed] (c2) -- (c3); 
      \draw (c3) -- (c4);
      \draw (c1) -- node[midway, below, sloped] {4} (c5);
    \end{tikzpicture}} \vspace{-30pt} \\ 
    $D_n:$ & \centertikz{\begin{tikzpicture}[
      thick,
      acteur/.style={
        circle,
        fill=black,
        thick,
        inner sep=2pt,
        minimum size=0.2cm
      },
      baseline={(0,-0.75)}
    ] 
      \node (a1) at (1,2) [acteur,label=$s_{n-2}$]{};
      \node (a2) at (0,2)[acteur]{}; 
      \node (a3) at (-1,2) [acteur,label=$s_2$]{}; 
      \node (a4) at (-2,2) [acteur,label=$s_1$]{}; 
      \node (a5) at (2,2.5) [acteur,label=above:$s_{n-1}$]{}; 
      \node (a6) at (2,1.5) [acteur,label=below:$s_{n}$]{};
      
      \draw (a1) -- (a2); 
      \draw[dashed] (a2) -- (a3); 
      \draw (a3) -- (a4);
      \draw (a1) -- (a5);
      \draw (a1) -- (a6);
      
      %\draw (a5)--(a6);
%     \draw (a7) -- (a5);
%     \draw[red] (a9) -- (a5);
%     \draw[red] (a9) -- (a4);
%      \draw[red] (a5) -- (a3);
%      \draw[red] (a6) -- (a2);

    \end{tikzpicture}}  & 
    $I_2(n):$ & \centertikz{\begin{tikzpicture}[
      thick,
      acteur/.style={
        circle,
        fill=black,
        thick,
        inner sep=2pt,
        minimum size=0.2cm
      },
      baseline={(0,0.5)}
    ]         
        \node (t1) at (0,0)[acteur,label=below:$s_1$]{};
        \node (t2) [right of=t1][acteur,label=below:$s_2$]{};

        \draw (t1)--node[midway, above, sloped] {$n$}(t2);            
    \end{tikzpicture}}
   \end{tabular} 
   \vspace{-40pt} % adjust this value as needed
   \caption{Coxeter graphs of finite Coxeter groups} 
   \label{tab:groupstable} 
\end{table} 
For the folding subgroups in part (i), the partition $\widehat{S}$ comes from the orbits of the unique nontrivial automorphism on the Coxeter graph of $W$. Such an automorphism induces a bijection $\sigma: S\rightarrow S$ which extends uniquely to an automorphism of $W$ that preserves the length. In this case, the folding subgroup $\widehat{W}=\langle R\rangle$ corresponds to the subgroup of $W$ consisting of the elements fixed by $\sigma$.  We refer to this type of folding as a \textbf{classical folding}. 
 For a reference on classical foldings and their associated root systems, see \cite{steinberg1967lectures}.  We remark that the foldings found in parts (iii) and (vi) in Table \ref{FiniteFoldingTable} are also classical foldings.  \\

The embeddings in part (ii) corresponds to a Lusztig partition of $S$ that is not induced by an automorphism on the Coxeter graph. We refer to this type of folding as a \textbf{Lusztig folding}.  The foldings found in parts (iv) and (v) in Table \ref{FiniteFoldingTable} are also examples of Lusztig foldings.  For more details on Lusztig foldings, see \cite{lusztig1983some}.\\

In Table \ref{DescriptionFiniteFoldingTable}, we give a detailed description of each of the foldings found in Table \ref{FiniteFoldingTable} parts (i) and (ii). If $\widehat W$ and $W$ are Coxeter groups of rank $n$ and $m$, we use $R=\{r_1,\ldots,r_n\}$ and $S=\{s_1,\ldots,s_m\}$ to denote their generating sets. \\

%In Table \ref{DescriptionFiniteFoldingTable}, Table %\ref{DescriptionAffineFoldingTable1}, and Table %\ref{DescriptionAffineFoldingTable2}, we also include diagrams of how the %Coxeter graphs "fold" in relation to the set partition $\tilde S$ of $S$.\\

\begin{table}[h!]
    \centering
    \begin{tblr}{|c|c|c|}
    \hline
    $\phi:\widehat{W}\to W$ & $\widehat{W}$ & $W$ \\
    \hline
    
    \SetCell[r=4]{c} $\phi:B_n\to A_{2n-1}$ & \SetCell[r=12]{c}\begin{tikzpicture}[
      thick,
      acteur/.style={
        circle,
        fill=black,
        thick,
        inner sep=2pt,
        minimum size=0.2cm
      }
    ] 
      
\node (c1) at (7.25,2) [acteur,label=below:$r_{n-1}$]{};
      \node (c2) at (6.5,2)[acteur]{}; 
      \node (c3) at (5.75,2) [acteur,label=below:$r_2$]{}; 
      \node (c4) at (5,2) [acteur,label=below:$r_1$]{}; 
      \node (c5) at (8,2) [acteur,label=below:$r_n$]{}; 
%      \node (a6) at (2.5,0.5) [acteur,label=below:$s_{n+1}$]{};
%      \node (a7) at (1,0.5) [acteur]{}; 
%      \node (a8) at (-0.5,0.5) [acteur,label=below:$s_{2n-2}$]{}; 
%      \node (a9) at (-2,0.5) [acteur,label=below:$s_{2n-1}$]{}; 
      \draw (c1) -- (c2); 
      \draw[dashed] (c2) -- (c3); 
      \draw (c3) -- (c4);
      \draw (c1) -- node[midway, above, sloped] {4} (c5);
    \end{tikzpicture}  & \SetCell[r=4]{c} \begin{tikzpicture}[
      thick,
      acteur/.style={
        circle,
        fill=black,
        thick,
        inner sep=2pt,
        minimum size=0.2cm
      }
    ] 
      \node (a1) at (1,2.5) [acteur,label=$s_{n-1}$]{};
      \node (a2) at (0,2.5)[acteur]{}; 
      \node (a3) at (-1,2.5) [acteur,label=$s_2$]{}; 
      \node (a4) at (-2,2.5) [acteur,label=$s_1$]{}; 
      \node (a5) at (2,2) [acteur,label=below:$s_n$]{}; 
      \node (a6) at (1,1.5) [acteur,label=below:$s_{n+1}$]{};
      \node (a7) at (0,1.5) [acteur]{}; 
      \node (a8) at (-1,1.5) [acteur,label=below:$s_{2n-2}$]{}; 
      \node (a9) at (-2,1.5) [acteur,label=below:$s_{2n-1}$]{}; 
      \draw (a1) -- (a2); 
      \draw[dashed] (a2) -- (a3); 
      \draw (a3) -- (a4);
      \draw (a1) -- (a5);
      \draw (a5) -- (a6);
      
      \draw (a6) -- (a7);
      \draw[dashed] (a7) -- (a8);
      \draw (a8) -- (a9);
      \draw[blue, <->] (a4)--(a9);
      \draw[blue, <->] (a3)--(a8);
      \draw[blue, <->] (a2)--(a7);
      \draw[blue, <->] (a1)--(a6);
%     \draw (a7) -- (a5);
%     \draw[red] (a9) -- (a5);
%     \draw[red] (a9) -- (a4);
%      \draw[red] (a5) -- (a3);
%      \draw[red] (a6) -- (a2);

    \end{tikzpicture} \\
    $\phi(r_{i})=s_is_{2n-i}$ &&\\
    $i\neq n$ &&\\    
    $\phi(r_{n})=s_n$ && \\ \hline
    $\phi:B_n\to A_{2n}$& & \SetCell[r=4]{c} \begin{tikzpicture}[
      thick,
      acteur/.style={
        circle,
        fill=black,
        thick,
        inner sep=2pt,
        minimum size=0.2cm
      }
    ] 
      \node (a1) at (2,2.5) [acteur,label=$s_{n}$]{};
      \node (a2) at (1,2.5)[acteur]{}; 
      \node (a3) at (0,2.5) [acteur,label=$s_2$]{}; 
      \node (a4) at (-1,2.5) [acteur,label=$s_1$]{}; 
      %\node (a5) at (2,2) [acteur,label=below:$s_n$]{}; 
      \node (a6) at (2,1.5) [acteur,label=below:$s_{n+1}$]{};
      \node (a7) at (1,1.5) [acteur]{}; 
      \node (a8) at (0,1.5) [acteur,label=below:$s_{2n-1}$]{}; 
      \node (a9) at (-1,1.5) [acteur,label=below:$s_{2n}$]{}; 
      \draw (a1) -- (a2); 
      \draw[dashed] (a2) -- (a3); 
      \draw (a3) -- (a4);
      \draw (a1)edge[bend left=80] (a6);
      %\draw (a5) -- (a6);
      
      \draw (a6) -- (a7);
      \draw[dashed] (a7) -- (a8);
      \draw (a8) -- (a9);

      \draw[blue, <->] (a4)--(a9);
      \draw[blue, <->] (a3)--(a8);
      \draw[blue, <->] (a2)--(a7);
      \draw[blue, <->] (a1)--(a6);
%     \draw (a7) -- (a5);
%     \draw[red] (a9) -- (a5);
%     \draw[red] (a9) -- (a4);
%      \draw[red] (a5) -- (a3);
%      \draw[red] (a6) -- (a2);

    \end{tikzpicture}\\
    $\phi(r_{i})=s_is_{2n+1-i}$ &&\\
    $i\neq n$ &&\\
    $\phi(r_{n})=s_ns_{n+1}s_n$ &&\\ \hline
    $\phi:B_n\to D_{n+1}$ & & \SetCell[r=4]{c} \begin{tikzpicture}[
      thick,
      acteur/.style={
        circle,
        fill=black,
        thick,
        inner sep=2pt,
        minimum size=0.2cm
      }
    ] 
      \node (a1) at (1,2) [acteur,label=$s_{n-1}$]{};
      \node (a2) at (0,2)[acteur]{};
      
      \node (a3) at (-1,2) [acteur,label=$s_2$]{}; 
      \node (a4) at (-2,2) [acteur,label=$s_1$]{}; 
      \node (a5) at (2,2.5) [acteur,label=above:$s_n$]{}; 
      \node (a6) at (2,1.5) [acteur,label=below:$s_{n+1}$]{};
      
      \draw (a1) -- (a2); 
      \draw[dashed] (a2)--(a3);
      \draw (a3) -- (a4);
      \draw (a1) -- (a5);
      \draw (a1) -- (a6);
      
      \draw[blue, <->] (a5)--(a6);
%     \draw (a7) -- (a5);
%     \draw[red] (a9) -- (a5);
%     \draw[red] (a9) -- (a4);
%      \draw[red] (a5) -- (a3);
%      \draw[red] (a6) -- (a2);

    \end{tikzpicture}\\
    $\phi(r_i)=s_i$ &&\\
    $i\neq n$ &&\\
    $\phi(r_n)=s_ns_{n+1}$ &&\\ \hline
    $\phi:I_2(n+1)\to A_n$ & \SetCell[r=3]{c} \begin{tikzpicture}[
      thick,
      acteur/.style={
        circle,
        fill=black,
        thick,
        inner sep=2pt,
        minimum size=0.2cm
      }
    ]    

        \node (t1) at (0,0)[acteur,label=below:$r_1$]{};
        \node (t2) [right of=t1][acteur,label=below:$r_2$]{};

        \draw (t1)--node[midway, above, sloped] {$n+1$}(t2);

    \end{tikzpicture} & \SetCell[r=3]{c} \begin{tikzpicture}[
      thick,
      acteur/.style={
        circle,
        fill=black,
        thick,
        inner sep=2pt,
        minimum size=0.2cm
      }
    ] 
        \node (a1) at (-2,2.5) [acteur,label=left:$s_1$]{};
        \node (a2) at (-1.25,2.0) [acteur,label=right:$s_2$]{};
        \node (a3) at (-2,1.5) [acteur,label=left:$s_3$]{};
        \node (a4) at (-1.25,1.0) [acteur,label=right:$s_4$]{};
        \node (a5) at (-2,0.5) [acteur,label=left:$s_{n-1}$]{};
        \node (a6) at (-1.25,0.0) [acteur,label=right:$s_n$]{};

        \draw (a1)--(a2)
        (a2)--(a3)
        (a3)--(a4)
        (a5)--(a6);
        \draw[densely dashed](a4)--(a5);
        \draw[blue, <->] (a1)--(a3);\draw[blue,densely dashed, <->] (a3)--(a5);
        \draw[blue, <->] (a2)--(a4);\draw[blue,densely dashed, <->] (a4)--(a6);

        \node (a1) at (1,2.5) [acteur,label=left:$s_1$]{};
        \node (a2) at (1.75,2.0) [acteur,label=right:$s_2$]{};
        \node (a3) at (1,1.5) [acteur,label=left:$s_3$]{};
        \node (a4) at (1.75,1.0) [acteur,label=right:$s_4$]{};
        \node (a5) at (1,0.5) [acteur,label=left:$s_{n-2}$]{};
        \node (a7) at (1.75,0) [acteur,label=right:$s_{n-1}$]{};
        \node (a6) at (1,-0.5) [acteur,label=left:$s_n$]{};

        \draw (a1)--(a2)
        (a2)--(a3)
        (a3)--(a4)
        (a5)--(a7)--(a6);
        \draw[densely dashed](a4)--(a5);
        \draw[blue, <->] (a1)--(a3);\draw[blue,densely dashed, <->] (a3)--(a5);\draw[blue, <->] (a6)--(a5);
        \draw[blue, <->] (a2)--(a4);\draw[blue,densely dashed, <->] (a4)--(a7);
      
    \end{tikzpicture}\\
    $\displaystyle\phi(r_1)= \prod_{\substack{s_i\in S\\i\text{ odd}}}s_{i}$ &&\\
    $\displaystyle\phi(r_2)=\prod_{\substack{s_i\in S\\i\text{ even}}}s_i$ && \\
    \hline
\end{tblr}
\vspace{10pt}
    \caption{Description of the finite folding (i) and (ii) in Table \ref{FiniteFoldingTable}}
    \label{DescriptionFiniteFoldingTable}
\end{table}
For $k\in\Z_{>0}$, define the $q$-integer
$$[k]_q:=1+q+q^2+\cdots+q^{k-1}.$$ 

%We now state our main result for the finite foldings. 

\begin{theorem} %[Main Theorem for finite Coxeter foldings]
Consider the Coxeter groups of type $A,B,D,I_2(n)$ in relation to the foldings given in  Table \ref{DescriptionFiniteFoldingTable}.  The following are true:

%For simply-laced Coxeter groups of type $A$ and $D$, %let $\ell_A$ and $\ell_D$ denote their corresponding %length functions. 
\begin{enumerate}
    \item $U_{B_n}^{A_{2n-1}}(q)=\displaystyle\prod_{k=1}^{2n}\ [k]_{(-1)^kq}$\\
    
    \item $U_{B_n}^{A_{2n}}(q)=\displaystyle\prod_{k=1}^{2n+1}\ [k]_{(-1)^kq}$\\
    
    \item $U_{B_n}^{D_{n+1}}(q)=\displaystyle[2]_q[3]_{-q}[4]_q\prod_{k=3}^{n}\left([2k+1]_q-q^k\right)$\\
    
    \item If $n=2m$, then $\displaystyle U_{I_2(n+1)}^{A_{n}}(q)=[2]_{q^m}[n+1]_{q^m}$\\
    
    \item If $n=2m-1$, then $\displaystyle U_{I_2(n+1)}^{A_{n}}(q)=[2]_{q^{m-1}}[2]_{q^m}[m]_{q^n}$

    %\item $\displaystyle U_{I_2(2n+1)}^{A_{2n}}(q)=[2]_{q^n}[2n+1]_{q^n}$
    %\item $\displaystyle U_{I_2(2n)}^{A_{2n-1}}(q)=[2]_{q^{n-1}}[2]_{q^n}[n]_{q^{2n-1}}$
\end{enumerate}
    \label{thm:MainTheorem}
\end{theorem}

%=[2]_q[3]_{-q}[4]_q\cdots [2n-1]_{-q}[2n]_q
%=[2]_q[3]_{-q}[4]_q\cdots [2n-1]_{-q}[2n]_q[2n+1]_{-q}

\begin{example}
    \label{ex:ClassicalFoldingB2}
    Consider $B_2$ with generators $R=\{r_1,r_2\}$ as the folding subgroup of $A_3$ with generators $S=\lbrace s_1,s_2,s_3\rbrace$ under the embedding $$\phi(r_1)=s_1s_3,\quad \phi(r_2)=s_2.$$ 
    
    %or as the folding subgroup of $A_4$ under the embedding %$$\phi(r_1)=s_1s_4,\quad \phi(r_2)=s_2s_3s_2.$$   
    %In Table \ref{TableA34B2}, we list the unfolding of each element in $B_2$ as elements in $A_3$ and $A_4$, and their respective length in each case. 
    
    In Figure \ref{BruhatOrderA3}, we show the distribution of the unfolding of the elements in $B_2$ in the Bruhat order of $A_3$.
By Theorem \ref{thm:MainTheorem} we have
   \begin{equation*}
       U_{B_2}^{A_3}(q)=1+q+q^2+2q^3+q^4+q^5+q^6=[2]_q[3]_{-q}[4]_q
       \label{eq:UnfoldingPolynomialB2A3Example}
   \end{equation*}
   % and
   % \begin{equation*}
   %     U_{B_2}^{A_4}(q)=1+q^2+q^3+2q^5+q^7+q^8+q^{10}=[2]_q[3]_{-q}[4]_q[5]_{-q}
   %     \label{eq:UnfoldingPolynomialB2A4Example}
   % \end{equation*}
%Note that these formulas can be seen directly from Table \ref{TableA34B2}.
%    \label{ex:UnfoldinPolynomialB2Example}    
\end{example}

\begin{comment}
\begin{center}
    \begin{table}[h!]
        \centering
        \begin{tabular}{|cc|cc|cc|} \hline
          Word $B_2$ & Length $B_2$ & Word $A_3$ & Length $A_3$  & Word $A_4$ & Length $A_4$\\ \hline
           $e$ & $0$ & $e$ & $0$ & $e$ & $0$  \\
           $r_1$ & $1$ & $s_1s_3$ & $2$ & $s_1s_4$ & $2$ \\
           $r_2$ & $1$ & $s_2$ & $1$ & $s_2s_3s_2$ & $3$ \\        
           $r_2r_1$ & $2$ & $s_2s_1s_3$ & $3$ & $s_2s_3s_2s_1s_4$ &  $5$ \\
           $r_1r_2$ & $2$ & $s_1s_3s_2$ & $3$ & $s_1s_4s_2s_3s_2$ & $5$\\
           $r_1r_2r_1$ & $3$ & $s_1s_3s_2s_1s_3$ & $5$ & $s_1s_4s_2s_3s_2s_1s_4$ & $7$ \\
           $r_2r_1r_2$ & $3$ & $s_2s_1s_3s_2$ & $4$ & $s_2s_3s_2s_1s_4s_2s_3s_2$ & $8$ \\          
           $r_1r_2r_1r_2$ & $4$ & $s_1s_3s_2s_1s_3s_2$ & $6$ & $s_1s_4s_2s_3s_2s_1s_4s_2s_3s_2$ & $10$\\ \hline
        \end{tabular}
        \caption{Elements in $B_2$ and their respective unfolding in $A_3$ and $A_4$}
        \label{TableA34B2}
    \end{table}
\end{center}
\end{comment}

\begin{figure}[h!]
    \centering 
    \begin{tikzpicture}[auto,node distance=1.9cm]
      \node (A) {\underline{\textcolor{red}{$s_1s_3s_2s_1s_3s_2$}}};
      
      \node (B) [below of=A] {$s_1s_2s_3s_1s_2$};
\node (C) [right of=B] {$s_2s_3s_2s_1s_2$};
\node (D) [left of=B] {\underline{\textcolor{red}{$s_1s_3s_2s_1s_3$}}};

\node (E) [below of=B] {$s_2s_3s_2s_1$};
\node (F) [right of=E] {\underline{\textcolor{red}{$s_2s_1s_3s_2$}}};
\node (G) [right of=F] {$s_3s_1s_2s_1$};
\node (H) [left of=E] {$s_1s_2s_3s_1$};
\node (I) [left of=H] {$s_1s_2s_3s_2$};

\node (J) [below right of=E] {$s_3s_2s_1$};
\node (K) [right of=J] {\underline{\textcolor{red}{$s_1s_3s_2$}}};
\node (L) [right of=K] {$s_1s_2s_1$};
\node (M) [below left of=E] {\underline{\textcolor{red}{$s_2s_1s_3$}}};
\node (N) [left of=M] {$s_2s_3s_2$};
\node (O) [left of=N] {$s_1s_2s_3$};

\node (P) [below left of=J] {\underline{\textcolor{red}{$s_1s_3$}}};
\node (Q) [right of=P] {$s_2s_1$};
\node (R) [right of=Q] {$s_1s_2$};
\node (S) [left of=P] {$s_3s_2$};
\node (T) [left of=S] {$s_2s_3$};

\node (U) [below of=P] {\underline{\textcolor{red}{$s_2$}}};
\node (V) [right of=U] {$s_1$};
\node (W) [left of=U] {$s_3$};

\node (X) [below of=U] {\underline{\textcolor{red}{$e$}}};

\draw (A) -- (B)
(A) -- (C)
(A) -- (D)
(C) -- (E)
(C) -- (F)
(C) -- (G)
(B) -- (F)
(B) -- (H)
(B) -- (I)
(D) -- (E)
(D) -- (G)
(D) -- (H)
(D) -- (I)
(G) -- (J)
(G) -- (K)--(F)
(G) -- (L)
(F) -- (L)
(F) -- (M)
(F) -- (N)
(E) -- (J)
(E) -- (M)
(E) -- (N)
(H) -- (L)
(H) -- (M)
(H) -- (O)
(I) -- (K)
(I) -- (N)
(I) -- (O)
(L) -- (Q)
(L) -- (R)
(K) -- (P)
(K) -- (R)
(K) -- (S)
(J) -- (P)
(J) -- (Q)
(J) -- (S)
(M) -- (P)
(M) -- (Q)
(M) -- (T)
(N) -- (S)
(N) -- (T)
(O) -- (P)
(O) -- (R)
(O) -- (T)
(V) -- (P)
(V) -- (Q)
(V) -- (R)
(U) -- (Q)
(U) -- (R)
(U) -- (S)
(U) -- (T)
(W) -- (P)
(W) -- (S)
(W) -- (T)
(X) -- (U)
(X) -- (V)
(X) -- (W);

    \end{tikzpicture} 
    \caption{Bruhat order in $A_3$ with the elements in $B_2$ shown in red.}
    \label{BruhatOrderA3}
  \end{figure}
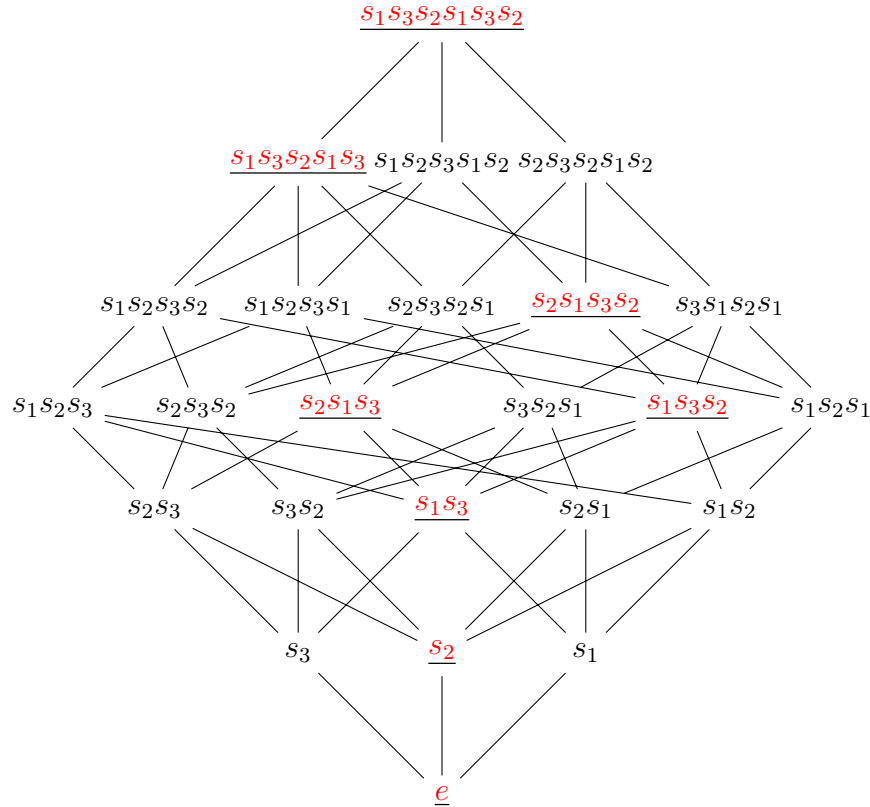

As a consequence of Theorem \ref{thm:MainTheorem}, we establish the following identity that relates the elements in $A_{n}$ with the elements in the folding subgroup $B_{\lfloor(n+1)/2\rfloor}$ and their respective lengths.\\

\begin{coro}\label{cor:length_identity}
Consider $B_{\lfloor(n+1)/2\rfloor}$ as a folding subgroup of $A_n$ as in either Theorem \ref{thm:MainTheorem} part (1) or (2).  Let $\ell_A$ and $\ell_B$ denote the length functions of type $A$ and $B$.  Then
\begin{eqnarray*}
    \sum_{(v,w)\in B_{\lfloor(n+1)/2\rfloor}\times B_{\lfloor(n+1)/2\rfloor}}(-1)^{\ell_B(w)}q^{\ell_B(w)+\ell_A(\phi(v))} & = & \sum_{(x,y)\in A_{n}\times B_{\lfloor(n+1)/2\rfloor}}(-1)^{\ell_A(x)}q^{\ell_A(x)+\ell_B(y)}\\
\end{eqnarray*}    
\end{coro}

\begin{example}
    Let $n=3$, so we have the embedding $\phi: B_2\to A_3$. Then

    \begin{eqnarray*}
    \sum_{(v,w)\in B_2\times B_2}(-1)^{\ell_B(w)}q^{\ell_B(w)+\ell_A(\phi(v))} & = & \sum_{(x,y)\in A_{3}\times B_2}(-1)^{\ell_A(x)}q^{\ell_A(x)+\ell_B(y)}\\
    & = & 1-q+q^2-2q^4+2q^5-2q^6+q^8-q^9+q^{10}
\end{eqnarray*}

Note that the summation sets $|B_2\times B_2|=64$ and $|A_3\times B_2|=192$.
\end{example}

%but the summation accounts only for $12$ of the %elements in those products, with no cancellation that %can be attributed to the nature of the pairs making %possible such an astonishing reduction.

%\begin{example}
%    Let $n=4$, so we have the embedding $\phi: B_2\to %A_4$. Then

%    \begin{eqnarray*}
%    \sum_{(v,w)\in B_2\times B_2}%(-1)^{\ell_B(w)}q^{\ell_B(w)+\ell_A(\phi(v))} & = & %\sum_{(x,y)\in A_{4}\times B_2}%(-1)^{\ell_A(x)}q^{\ell_A(x)+\ell_B(y)}\\
%    & = & 1-2q+3q^2-3q^3+q^4+2q^5-5q^6+6q^7\\
%    & & %-5q^8+2q^9+q^{10}-3q^{11}+3q^{12}-2q^{13}+q^{14}
%\end{eqnarray*}
%Note that $|B_2\times B_2|=64$ and $|A_4\times %B_2|=960$, but the summation accounts only for $40$ of %the elements in those products, with no cancellation %that can be attributed to the nature of the pairs %making possible such an astonishing reduction.
%\end{example}

\subsection{Results for foldings of Coxeter groups of affine type}
We consider folding subgroups of irreducible affine Coxeter groups.  Affine Coxeter groups are derived by adding a distinguished generator to a finite Coxeter group of the corresponding type and satisfy the property that every proper parabolic subgroup is a finite Coxeter group.  Affine Coxeter groups correspond to Weyl groups of affine Lie algebras and play an important role in their representation theory.  As with the finite case, irreducible affine folding pairs were classified by Crisp in \cite{crisp1996injective}.  We list them in Table \ref{AffineFoldingTable} and use the notational conventions on the classification of affine Coxeter groups found in \cite[Appendix A1]{bjorner2005combinatorics}. 

\begin{table}[H]
    \centering
    \begin{tabular}{ccc}    
         & $\qquad\qquad\widehat{W}\qquad\qquad$ & $\qquad W\qquad$ \\ \hline
         & & \vspace{-10pt} \\
         (i) & $\widetilde{A}_{n-1}$ & $\widetilde{A}_{mn-1}$, $m\geq 2$\\
        (ii) & $\widetilde{B}_n$ & $\widetilde{D}_{n+1}$, $\widetilde{D}_{2n}$, $\widetilde{D}_{2n+1}$\\
        (iii) & $\widetilde{C}_n$ & $\widetilde{A}_{2n-1}$, $\widetilde{A}_{2n}$, $\widetilde{A}_{2n+1}$, $\widetilde{B}_{n+1}$, $\widetilde{D}_{n+2}$, $\widetilde{C}_{2n}$, $\widetilde{C}_{2n+1}$\\
        %(iv) & $\tilde{C}_n$ & $\tilde{A}_{2n-1}$, $\tilde{A}_{2n}$, $\tilde{A}_{2n+1}$, $\tilde{D}_{n+2}$\\
        (iv) & $\widetilde{C}_2$ & $\widetilde{D}_5$\\
        (v) & $\widetilde{F}_4$ & $\widetilde{E}_6$, $\widetilde{E}_7$\\
        (vi) & $\widetilde{G}_2$ & $\widetilde{B}_3$, $\widetilde{F}_4$, $\widetilde{D}_4$, $\widetilde{D}_6$, $\widetilde{E}_6$, $\widetilde{E}_7$        
    \end{tabular}
    \caption{Coxeter embeddings of affine Coxeter groups}
    \label{AffineFoldingTable}
\end{table}

We focus on the infinite families of foldings found in parts (i)-(iii) of Table \ref{AffineFoldingTable} (note that parts (iv)-(vi) are not infinite families).  The Coxeter groups involved in these families are listed in Table \ref{tab:affinegroupstable}.  We separate these foldings pairs by the associated folding subgroup. \\

\begin{table}[H] 
   \centering 
   \begin{tabular}{
>{\centering\arraybackslash}m{0.03\linewidth}
>{\centering\arraybackslash}m{0.35\linewidth}
>{\centering\arraybackslash}m{0.03\linewidth}
>{\centering\arraybackslash}m{0.28\linewidth}
}    
   $\widetilde A_n:$  \vspace{-15pt} & \centertikz{\begin{tikzpicture}[ 
     thick, 
     acteur/.style={ 
       circle, 
       fill=black, 
       thick, 
       inner sep=2pt, 
       minimum size=0.2cm 
     } ,
      baseline={(0,-0.75)}
   ]  
     \node (a1) at (2,2.5) [acteur,label=$s_{k}$]{}; 
     \node (a2) at (1,2.5)[acteur]{};  
     \node (a3) at (0,2.5) [acteur,label=$s_1$]{};  
     \node (a4) at (-1,2.5) [acteur,label=$s_0$]{};  
     %\node (a5) at (2,2) [acteur,label=below:$s_n$]{};  
     \node (a6) at (2,1.5) [acteur,label=below:$s_{k+1}$]{}; 
     \node (a7) at (1,1.5) [acteur]{};  
     \node (a8) at (0,1.5) [acteur,label=below:$s_{n-1}$]{};  
     \node (a9) at (-1,1.5) [acteur,label=below:$s_{n}$]{};  
     \draw (a1) -- (a2);  
     \draw[dashed] (a2) -- (a3);  
     \draw (a3) -- (a4); 
     \draw (a1)edge[bend left=80] (a6); 
     \draw (a4)edge[bend right=80] (a9); 
     %\draw (a5) -- (a6); 
      
     \draw (a6) -- (a7); 
     \draw[dashed] (a7) -- (a8); 
     \draw (a8) -- (a9); 
 
     %\draw[blue, <->] (a4)--(a9); 
     %\draw[blue, <->] (a3)--(a8); 
     %\draw[blue, <->] (a2)--(a7); 
     %\draw[blue, <->] (a1)--(a6); 
%     \draw (a7) -- (a5); 
%     \draw[red] (a9) -- (a5); 
%     \draw[red] (a9) -- (a4); 
%      \draw[red] (a5) -- (a3); 
%      \draw[red] (a6) -- (a2); 
 
   \end{tikzpicture}} \vspace{-30pt}
   &
   $\widetilde B_n:$  \vspace{-15pt}& \centertikz{\begin{tikzpicture}[ 
     thick, 
     acteur/.style={ 
       circle, 
       fill=black, 
       thick, 
       inner sep=2pt, 
       minimum size=0.2cm 
     } ,
      baseline={(0,-0.75)}
   ]  
     \node (a1) at (2,2) [acteur,label=below:$s_{n-2}$]{}; 
     \node (a2) at (1,2)[acteur]{};  
     \node (a3) at (-1,2) [acteur,label=below:$s_1$]{};  
     \node (a4) at (-2,2) [acteur,label=below:$s_0$]{};  
     \node (a5) at (3,2.5) [acteur,label=above:$s_{n-1}$]{};  
     \node (a6) at (3,1.5) [acteur,label=below:$s_{n}$]{}; 
     \node (a7) at (0,2) [acteur,label=below:$s_2$]{}; 
     \draw (a1) -- (a2);  
     \draw[dashed] (a7) -- (a2);  
     \draw (a7)--(a3) -- node[midway, above, sloped] {4} (a4); 
     \draw (a1) -- (a5); 
     \draw (a1) -- (a6);

     %\draw[blue, <->] (a5)--(a6); 
%     \draw (a7) -- (a5); 
%     \draw[red] (a9) -- (a5); 
%     \draw[red] (a9) -- (a4); 
%      \draw[red] (a5) -- (a3); 
%      \draw[red] (a6) -- (a2); 
 
   \end{tikzpicture}}  \vspace{-30pt}\\ 
$\widetilde C_n:$ & \centertikz{\begin{tikzpicture}[ 
     thick, 
     acteur/.style={ 
       circle, 
       fill=black, 
       thick, 
       inner sep=2pt, 
       minimum size=0.2cm 
     } 
   ]  
      
\node (c1) at (9,2) [acteur,label=below:$s_{n-1}$]{}; 
     \node (c2) at (8,2)[acteur]{};  
     \node (c3) at (6,2) [acteur,label=below:$s_1$]{};  
     \node (c4) at (5,2) [acteur,label=below:$s_0$]{};  
     \node (c5) at (10,2) [acteur,label=below:$s_n$]{};  
   \node (c6) at (7,2) [acteur,label=below:$s_{2}$]{}; 
   %\node (c7) at (8,2) [acteur]{};  
%      \node (a8) at (-0.5,0.5) [acteur,label=below:$s_{2n-2}$]{};  
%      \node (a9) at (-2,0.5) [acteur,label=below:$s_{2n-1}$]{};  
     \draw (c1) -- (c2);  
     \draw[dashed] (c6) -- (c2);  
     \draw (c6)--(c3) --node[midway, above, sloped] {4} (c4); 
     \draw (c1) -- node[midway, above, sloped] {4} (c5); 
   \end{tikzpicture}}
   &
   $\widetilde D_n:$ & \centertikz{\begin{tikzpicture}[ 
     thick, 
     acteur/.style={ 
       circle, 
       fill=black, 
       thick, 
       inner sep=2pt, 
       minimum size=0.2cm 
     } ,
      baseline={(0,-0.75)}
   ]  
     \node (a1) at (1,2) [acteur,label=$s_{n-2}$]{}; 
     \node (a2) at (0,2)[acteur]{};  
     \node (a3) at (-1,2) [acteur,label=$s_3$]{};  
     \node (a4) at (-2,2) [acteur,label=$s_2$]{};  
     \node (a5) at (2,2.5) [acteur,label=above:$s_{n-1}$]{};  
     \node (a6) at (2,1.5) [acteur,label=below:$s_{n}$]{}; 
     \node (a7) at (-3,2.5) [acteur,label=above:$s_{0}$]{};  
     \node (a8) at (-3,1.5) [acteur,label=below:$s_{1}$]{}; 
      
     \draw (a1) -- (a2);  
     \draw[dashed] (a2) -- (a3);  
     \draw (a3) -- (a4); 
     \draw (a1) -- (a5); 
     \draw (a1) -- (a6); 
     \draw (a7)--(a4)--(a8); 
      
     %\draw[blue, <->] (a5)--(a6); 
     %\draw (a7)--(a8); 
%     \draw (a7) -- (a5); 
%     \draw[red] (a9) -- (a5); 
%     \draw[red] (a9) -- (a4); 
%      \draw[red] (a5) -- (a3); 
%      \draw[red] (a6) -- (a2); 
 
   \end{tikzpicture}}
   \end{tabular} 
   \vspace{-45pt} % adjust this value as needed
   \caption{Coxeter graphs of affine Coxeter groups} 
   \label{tab:affinegroupstable} 
\end{table} 

%Similarly, Table \ref{DescriptionAffineFoldingTable1} and Table %\ref{DescriptionAffineFoldingTable2} give a detailed description of the foldings %(i) through (iv) in Table \ref{AffineFoldingTable}. 

%As a second main result, we calculate the unfolding series in terms of $q$-%integers for the eleven infinite families given in parts (i) through (iv) in %Table \ref{AffineFoldingTable} (note that parts (v)-(vii) are not infinite %families).\\

First, we consider the folding subgroup $\widetilde A_{n-1}$ in  $\widetilde A_{mn-1}$ given in part (i) of Table  \ref{AffineFoldingTable}.   This folding is explicitly described in Table \ref{tab:DescriptionAffineFoldingTableTypeA}.\\

\begin{table}[h]
    \centering
     \begin{tabular}{|c|c|c|}
\hline
\rule{0pt}{3ex}

  $\phi:\widehat{W}\to W$ & $\widehat{W}$ & $W$ \tabularnewline
\hline 
 & \multirow{7}{*}{
 \begin{tikzpicture}[
      thick,
      acteur/.style={
        circle,
        fill=black,
        thick,
        inner sep=2pt,
        minimum size=0.2cm
      }
    ]       
\node (c1) at (8,2) [acteur,label=below:$r_{n-2}$]{};
      \node (c2) at (7,2)[acteur]{}; 
      \node (c3) at (6,2) [acteur,label=below:$r_1$]{}; 
      \node (c4) at (5,2) [acteur,label=below:$r_0$]{}; 
      \node (c5) at (9,2) [acteur,label=below:$r_{n-1}$]{}; 
%      \node (a6) at (2.5,0.5) [acteur,label=below:$s_{n+1}$]{};
%      \node (a7) at (1,0.5) [acteur]{}; 
%      \node (a8) at (-0.5,0.5) [acteur,label=below:$s_{2n-2}$]{}; 
%      \node (a9) at (-2,0.5) [acteur,label=below:$s_{2n-1}$]{}; 
      \draw (c1) -- (c2); 
      \draw[dashed] (c2) -- (c3); 
      \draw (c3) -- (c4);
      \draw (c1) -- (c5);
      \draw (c4)edge[bend left=45](c5);
    \end{tikzpicture}} & \multirow{7}{*}{\begin{tikzpicture}[
      thick,
      acteur/.style={
        circle,
        fill=black,
        thick,
        inner sep=2pt,
        minimum size=0.2cm
      }
    ] 
      \node (a1) at (1,2.5) [acteur,label=$s_{2n-1}$]{};
      \node (a2) at (0,2.5)[acteur]{}; 
      \node (a3) at (-1,2.5) [acteur,label=$s_{n+1}$]{}; 
      \node (a4) at (-2,2.5) [acteur,label=$s_{n}$]{}; 
       
      \node (a6) at (1,1.5) [acteur,label=below:$s_{n-1}$]{};
      \node (a7) at (0,1.5) [acteur]{}; 
      \node (a8) at (-1,1.5) [acteur,label=below:$s_{1}$]{}; 
      \node (a9) at (-2,1.5) [acteur,label=below:$s_{0}$]{}; 

      \node (a10) at (-2,3.5) [acteur,label=above:$s_{(m-1)n}$]{};
      \node (a11) at (-1,3.5) [acteur]{};
      \node (a12) at (0,3.5) [acteur]{};
      \node (a13) at (1,3.5) [acteur,label=above:$s_{mn-1}$]{};

      \draw (a9)--(a8);
      \draw[dashed] (a8)--(a7);
      \draw (a7)--(a6)--(a4)--(a3);
      \draw[dashed] (a3)--(a2);
      \draw (a2)--(a1);
      \draw[dashed] (a1)--(a10);
      \draw (a10)--(a11);
      \draw[dashed] (a11)--(a12);
      \draw (a12)--(a13)--(a9);      
     
%     \draw (a7) -- (a5);
%     \draw[red] (a9) -- (a5);
%     \draw[red] (a9) -- (a4);
%      \draw[red] (a5) -- (a3);
%      \draw[red] (a6) -- (a2);

    \end{tikzpicture}}\tabularnewline
   &  & \tabularnewline
   $\phi:\widetilde A_{n-1}\to\widetilde A_{mn-1}$ &  & \tabularnewline
   $\phi(r_{i})=\displaystyle\prod_{j=0}^{m-1}s_{i+jn}$ &  & \tabularnewline
   &  & \tabularnewline
   &  & \tabularnewline
%& &  & \tabularnewline
\hline 
\end{tabular}
\vspace{10pt}
    \caption{Affine folding in Table \ref{AffineFoldingTable} part (i).}    \label{tab:DescriptionAffineFoldingTableTypeA}
\end{table}

\begin{theorem}\label{thm:affine_type_A}
    Let $\widetilde A_{n-1}$ be the folding subgroup of $\widetilde A_{mn-1}$ given by Table \ref{tab:DescriptionAffineFoldingTableTypeA}.  Then the corresponding unfolding series is
    $$U_{\widetilde A_{n-1}}^{\widetilde A_{mn-1}}(q)=\displaystyle\prod_{k=1}^n\frac{[k]_{q^m}}{1-q^{(k-1)m}}.$$
    \label{thm:UnfoldingSeriesTypeAffineA}
\end{theorem}

Next we consider folding pairs given in part (ii) of Table \ref{AffineFoldingTable}.  For these folding pairs, we realize $\widetilde B_n$ as a folding subgroup of a Coxeter group of affine type $D$ in three possible ways which we give in Table \ref{DescriptionAffineFoldingTable1}.

\begin{theorem}\label{thm:affine_type_B}
    Let $\widetilde B_n$ be a folding subgroup of a Coxeter group of affine type $D$ as in Table \ref{DescriptionAffineFoldingTable1}.  Then the corresponding unfolding series are, respectively, given by\\

    \begin{enumerate}
        \item $U_{\widetilde{B}_{n}}^{\widetilde{D}_{n+1}}(q)=\displaystyle\frac{[2]_q[3]_{-q}[4]_q}{(1-q)(1-q^3)(1+q^n)}\cdot\prod_{k=3}^n\frac{[2k+1]_q-q^k}{1-q^{2k-1}}$

%=\dfrac{[2]_{q}[3]_{-q}[4]_{q}\prod_{k=3}^{n}([2k+1]_{q}-q^k)}{(1-q)(1-q^3)\cdots(1-q^{2n-1})(1+q^n)}$\\
%$U_{\tilde{B}_{n}}^{\tilde{D}_{n+1}}(q)=\dfrac{U_{B_n}^{D_{n+1}}(q)}{(q^{n+1};q)_n}

\item $U_{\widetilde B_{n}}^{\widetilde D_{2n}}(q)=\displaystyle\left(\prod_{k=1}^{2n}\ [k]_{(-1)^kq}\right)\cdot\left(\prod_{k=1}^n\frac{1+q^{2(k-1)}}{1-q^{2(n+k)-3}}\right)$

%$\dfrac{[2]_{q}[3]_{-q}\cdots[2n-1]_{-q}[2n]_{q}\prod_{k=1}^{n-1}(1+q^{2k})}{(1-q^{2n-1})(1-q^{2n+1})\cdots(1-q^{4n-3})}$%\\
%$U_{\tilde B_{n}}^{\tilde D_{2n}}(q)=\dfrac{U_{B_n}^{A_{2n-1}}(q)\cdot\prod_{k=1}^{n-1}(1+q^{2k})}{(q^{2n-1};q^2)_n}=U_{B_n}^{A_{2n-1}}(q)\dfrac{(-q^2;q^2)_{n-1}}{(q^{2n-1};q^2)_n}$

\item $U_{\widetilde B_{n}}^{\widetilde D_{2n+1}}(q)=\displaystyle\left(\prod_{k=1}^{2n+1}\ [k]_{(-1)^kq}\right)\cdot\left(\prod_{k=1}^n\frac{1+q^{2(k-1)}}{1-q^{2(n+k)-1}}\right)$

%\dfrac{[2]_{q}[3]_{-q}\cdots[2n-1]_{-q}[2n]_{q}[2n+1]_{-q}\prod_{k=1}^{n-1}(1+q^{2k})}{(1-q^{2n+1})(1-q^{2n+3})\cdots(1-q^{4n-1})}$%\\
%$U_{\tilde B_{n}}^{\tilde D_{2n+1}}(q)=\dfrac{U_{B_n}^{A_{2n}}(q)\cdot\prod_{k=1}^{n-1}(1+q^{2k})}{(q^{2n+1};q^2)_n}=U_{B_n}^{A_{2n}}(q)\dfrac{(-q^2;q^2)_{n-1}}{(q^{2n+1};q^2)_n}$
    \end{enumerate}
    \label{thm:UnfoldingSeriesAffineTypeB}
\end{theorem}

\begin{table}[h!]
    \centering

     \begin{tabular}{|c|c|c|c|}
\hline \rule{0pt}{3ex}
 & $\phi:\widehat{W}\to W$ & $\widehat{W}$ & $W$ \tabularnewline
\hline 
& & \multirow{25}{*}{\begin{tikzpicture}[
      thick,
      acteur/.style={
        circle,
        fill=black,
        thick,
        inner sep=2pt,
        minimum size=0.2cm
      }
    ] 
      \node (a1) at (1,2) [acteur,label=below:$r_{n-2}$]{};
      \node (a2) at (0,2)[acteur]{}; 
      \node (a3) at (-1,2) [acteur,label=below:$r_1$]{}; 
      \node (a4) at (-2,2) [acteur,label=below:$r_0$]{};  
      \node (a5) at (2,2.5) [acteur,label=above:$r_{n-1}$]{}; 
      \node (a6) at (2,1.5) [acteur,label=below:$r_{n}$]{};
      
      \draw (a1) -- (a2); 
      \draw[dashed] (a2) -- (a3); 
      \draw (a3) -- node[midway, above, sloped] {4} (a4);
      \draw (a1) -- (a5);
      \draw (a1) -- (a6);

      %\draw[blue, <->] (a5)--(a6);
%     \draw (a7) -- (a5);
%     \draw[red] (a9) -- (a5);
%     \draw[red] (a9) -- (a4);
%      \draw[red] (a5) -- (a3);
%      \draw[red] (a6) -- (a2);

    \end{tikzpicture}} & \multirow{6}{*}{\begin{tikzpicture}[
      thick,
      acteur/.style={
        circle,
        fill=black,
        thick,
        inner sep=2pt,
        minimum size=0.2cm
      }
    ] 
      \node (a1) at (1,2) [acteur,label=$s_{n-1}$]{};
      \node (a2) at (0,2)[acteur]{}; 
      \node (a3) at (-1,2) [acteur,label=$s_3$]{}; 
      \node (a4) at (-2,2) [acteur,label=$s_2$]{}; 
      \node (a5) at (2,2.5) [acteur,label=above:$s_{n}$]{}; 
      \node (a6) at (2,1.5) [acteur,label=below:$s_{n+1}$]{};
      \node (a7) at (-3,2.5) [acteur,label=above:$s_{0}$]{}; 
      \node (a8) at (-3,1.5) [acteur,label=below:$s_{1}$]{};
      
      \draw (a1) -- (a2); 
      \draw[dashed] (a2) -- (a3); 
      \draw (a3) -- (a4);
      \draw (a1) -- (a5);
      \draw (a1) -- (a6);
      \draw (a7)--(a4)--(a8);
      
      %\draw[blue, <->] (a5)--(a6);
      \draw[blue, <->] (a7)--(a8);
%     \draw (a7) -- (a5);
%     \draw[red] (a9) -- (a5);
%     \draw[red] (a9) -- (a4);
%      \draw[red] (a5) -- (a3);
%      \draw[red] (a6) -- (a2);

    \end{tikzpicture}}\tabularnewline
  1 &  $\phi:\widetilde B_{n}\to\widetilde D_{n+1}$ &&\\
& $\phi(r_0)=s_0s_1$ &&\\
   & && \\
  &  $\phi(r_i)=s_{i+1}$ && \\
   & $i\neq 0$ && \\

\cline{1-1} \cline{2-2} \cline{4-4}
& &  & \multirow{9}{*}{\begin{tikzpicture}[
      thick,
      acteur/.style={
        circle,
        fill=black,
        thick,
        inner sep=2pt,
        minimum size=0.2cm
      }
    ] 
      \node (a1) at (1,2) [acteur,label=$s_{0}$]{};
      \node (a2) at (1,1)[acteur,acteur,label=$s_{1}$]{}; 
      \node (a3) at (0,1.5) [acteur,label=$s_2$]{}; 
      \node (a4) at (-1,1.5) [acteur,label=$s_3$]{}; 
      \node (a5) at (-2,1.5) [acteur,label=above:$s_{n-2}$]{}; 
      \node (a6) at (-3,1.5) [acteur,label=above:$s_{n-1}$]{};
      \node (a7) at (-4,0.5) [acteur,label=above:$s_{n}$]{}; 
      \node (a8) at (-3,-0.5) [acteur,label=below:$s_{n+1}$]{};
      \node (a9) at (-2,-0.5) [acteur,label=below:$s_{n+2}$]{};
      \node (a10) at (-1,-0.5) [acteur,label=below:$s_{2n-3}$]{};
      \node (a11) at (0,-0.5) [acteur,label=below:$s_{2n-2}$]{};
      \node (a12) at (1,-0) [acteur,label=below:$s_{2n-1}$]{};
      \node (a13) at (1,-1) [acteur,label=below:$s_{2n}$]{};
      
      \draw (a1) -- (a3)--(a4);
      \draw (a2)--(a3);
      \draw[dashed] (a4) -- (a5); 
      \draw (a5) -- (a6)--(a7)--(a8)--(a9);
      \draw[dashed] (a9) -- (a10);
      \draw (a10) -- (a11)--(a12);
      \draw (a11)--(a13);
      
      \draw[blue, <->] (a3)--(a11);
      \draw[blue, <->] (a4)--(a10);
      \draw[blue, <->] (a5)--(a9);
      \draw[blue, <->] (a6)--(a8);
      \draw[blue, <->] (a1)edge[bend left=45](a13);
      \draw[blue, <->] (a2)edge[bend left=45](a12);
%     \draw (a7) -- (a5);
%     \draw[red] (a9) -- (a5);
%     \draw[red] (a9) -- (a4);
%      \draw[red] (a5) -- (a3);
%      \draw[red] (a6) -- (a2);

    \end{tikzpicture}}\tabularnewline
 &&  & \tabularnewline
2 & $\phi:\widetilde B_{n}\to\widetilde D_{2n}$ &  & \tabularnewline
  &  $\phi(r_0)=s_{n}$&& \\
 & & & \tabularnewline
&$\phi(r_i)=s_is_{2n-i}$ &&\\
 &   $i\neq 0$ && \\
 & & & \tabularnewline

 &  && \tabularnewline
\cline{1-1} \cline{2-2} \cline{4-4} 
& &  & \multirow{9}{*}{\begin{tikzpicture}[
      thick,
      acteur/.style={
        circle,
        fill=black,
        thick,
        inner sep=2pt,
        minimum size=0.2cm
      }
    ] 
      \node (a1) at (1,2) [acteur,label=$s_{0}$]{};
      \node (a2) at (1,1)[acteur,acteur,label=$s_{1}$]{}; 
      \node (a3) at (0,1.5) [acteur,label=$s_2$]{}; 
      \node (a4) at (-1,1.5) [acteur,label=$s_3$]{}; 
      \node (a5) at (-2,1.5) [acteur,label=above:$s_{n-1}$]{}; 
      \node (a6) at (-3,1.5) [acteur,label=above:$s_{n}$]{};
      %\node (a7) at (-4,0.5) [acteur,label=above:$s_{n}$]{}; 
      \node (a8) at (-3,-0.5) [acteur,label=below:$s_{n+1}$]{};
      \node (a9) at (-2,-0.5) [acteur,label=below:$s_{n+2}$]{};
      \node (a10) at (-1,-0.5) [acteur,label=below:$s_{2n-2}$]{};
      \node (a11) at (0,-0.5) [acteur,label=below:$s_{2n-1}$]{};
      \node (a12) at (1,-0) [acteur,label=below:$s_{2n}$]{};
      \node (a13) at (1,-1) [acteur,label=below:$s_{2n+1}$]{};
      
      \draw (a1) -- (a3)--(a4);
      \draw (a2)--(a3);
      \draw[dashed] (a4) -- (a5); 
      \draw (a5) -- (a6)edge[bend right=45](a8);
      \draw (a8)--(a9);
      \draw[dashed] (a9) -- (a10);
      \draw (a10) -- (a11)--(a12);
      \draw (a11)--(a13);
      
      \draw[blue, <->] (a3)--(a11);
      \draw[blue, <->] (a4)--(a10);
      \draw[blue, <->] (a5)--(a9);
      \draw[blue, <->] (a6)--(a8);
      \draw[blue, <->] (a1)edge[bend left=45](a13);
      \draw[blue, <->] (a2)edge[bend left=45](a12);
%     \draw (a7) -- (a5);
%     \draw[red] (a9) -- (a5);
%     \draw[red] (a9) -- (a4);
%      \draw[red] (a5) -- (a3);
%      \draw[red] (a6) -- (a2);

    \end{tikzpicture}}\tabularnewline
 &&  & \tabularnewline
3 &$\phi:\widetilde B_{n}\to\widetilde D_{2n+1}$ &  & \tabularnewline
 & $\phi(r_0)=s_{n}s_{n+1}s_{n}$ && \\
 &  && \tabularnewline
& $\phi(r_i)=s_is_{2n+1-i}$ &&\\
  &  $i\neq 0$ && \\
 &  && \tabularnewline

 &&  & \tabularnewline
\hline 
\end{tabular}
\vspace{10pt} % adjust this value as needed
    \caption{Affine foldings from Table \ref{AffineFoldingTable} part (ii).}
    \label{DescriptionAffineFoldingTable1}
\end{table}

Observe that each folding in Theorem \ref{thm:affine_type_B} has an analogous folding in finite type.  In particular, parts 1,2, and 3 of Theorem \ref{thm:affine_type_B} correspond to the folding pairs $(B_n,D_{n+1})$, $(B_n,A_{2n-1})$ and $(B_n,A_{2n})$ respectively.  Moreover, the unfolding polynomial of finite type appears as factor in the corresponding unfolding series of affine type.\\

Finally, we consider the seven affine folding pairs in part (iii) of Table \ref{AffineFoldingTable}.  In each of these folding pairs, we realize $\widetilde{C}_{n}$ as a folding subgroup of an affine Coxeter group of type $A$, $B$ or $D$.  The details of these foldings are given in Table \ref{DescriptionAffineFoldingTable2}.

%Consider the Coxeter groups of type $\tilde A,\tilde %B,\tilde C,\tilde D$ in relation to the foldings given in  %Table \ref{DescriptionAffineFoldingTable1}, Table %\ref{DescriptionAffineFoldingTable2}, and %\ref{tab:DescriptionAffineFoldingTableTypeA}.

\begin{theorem}\label{thm:affine_type_C}
    Let $\widetilde C_n$ be a folding subgroup of a Coxeter group of affine type $A, B$ or $D$ as in Table \ref{DescriptionAffineFoldingTable2}.  Then the corresponding unfolding series are, respectively, given by

    \begin{enumerate}
        \item $U_{\widetilde{C}_{n}}^{\widetilde{A}_{2n+1}}(q)=\displaystyle[n+1]_{-q}[n+1]_{(-1)^nq}\cdot\left(\prod_{\substack{k=1\\k\neq n+1}}^{2n+1}\dfrac{[k]_{(-1)^kq}}{1+(-q)^k}\right)$              
        
%$[2]_{q}[3]_{-q}\cdots[2n]_q[2n+1]_{-q}[n+1]_{q}\displaystyle\prod_{\substack{i=1\\i\neq %n+1}}^{2n+1}\dfrac{1}{1+(-q)^i}$\\
%$U_{\tilde{C}_{n}}^{\tilde{A}_{2n+1}}(q)=U_{B_n}^{A_{2n}}(q)\cdot[n+1]_{-q}\displaystyle\prod_{\substack{i=1\\i\neq n+1}}^{2n+1}\dfrac{1}{1+(-q)^i}=U_{B_n}^{A_{2n}}(q)\dfrac{(-q^3;q^2)_n}{(q^{2n+4};q^2)_n}$

\item $U_{\widetilde{C}_n}^{\widetilde{A}_{2n}}(q)=\displaystyle \prod_{k=1}^{2n}\frac{[k+1]_{(-1)^{k+1}q}}{1+(-q)^k}$

%$\dfrac{[2]_{q}[3]_{-q}\cdots[2n-1]_{-q}[2n]_{q}[2n+1]_{-q}}{(1-q)(1+q^2)(1-q^3)(1+q^4)\cdots(1+q^{2n-2})(1-q^{2n-1})(1+q^{2n})}$\\
%$U_{\tilde{C}_n}^{\tilde{A}_{2n}}(q)=\dfrac{U_{B_n}^{A_{2n}}(q)}{(1-q)(1+q^2)(1-q^3)(1+q^4)\cdots(1+q^{2n-2})(1-q^{2n-1})(1+q^{2n})}=\dfrac{U_{B_n}^{A_{2n}}(q)}{(q;-q)_{2n}}$

\item $U_{\widetilde{C}_n}^{\widetilde{A}_{2n-1}}(q)=\displaystyle \prod_{k=2}^{2n}\frac{[k]_{(-1)^{k}q}}{1+(-q)^{k-1}}$
%\textcolor{red}{If $k=1$, the denominator becomes $2$. Start $k$ at $2$?}

%$=\dfrac{[2]_{q}[3]_{-q}\cdots[2n-1]_{-q}[2n]_{q}}{(1-q)(1+q^2)(1-q^3)(1+q^4)\cdots(1+q^{2n-2})(1-q^{2n-1})}$\\
%$U_{\tilde{C}_n}^{\tilde{A}_{2n-1}}(q)=\dfrac{U_{B_n}^{A_{2n-1}}(q)}{(1-q)(1+q^2)(1-q^3)(1+q^4)\cdots(1+q^{2n-2})(1-q^{2n-1})}=\dfrac{U_{B_n}^{2n-1}(q)}{(q;-q)_{2n-1}}$
 
 \item $U_{\widetilde{C}_{n}}^{\widetilde{B}_{n+1}}(q)=\displaystyle\frac{[2]_q[3]_{-q}[4]_q[2]_{-q^{n+1}}}{(1-q)(1-q^3)(1-q^5)}\cdot\prod_{k=3}^n\frac{[2k+1]_q-q^k}{1-q^{2k+1}}$ 
 
% $=\dfrac{[2]_q[3]_{-q}[4]_q[2]_{-q^{n+1}}\prod_{k=3}^{n}([2k+1]_q-q^k)}{(1-q)(1-q^3)\cdots (1-q^{2n+1})}$\\
 %$U_{\tilde{C}_{n}}^{\tilde{B}_{n+1}}(q)=\dfrac{U_{B_n}^{D_{n+1}}(q)\cdot[2]_{-q^{n+1}}}{(1-q)(1-q^3)\cdots (1-q^{2n+1})}=\dfrac{U_{B_n}^{D_{n+1}}(q)\cdot[2]_{-q^{n+1}}}{(q;q^2)_{n+1}}$ 

\item $U_{\widetilde C_{n}}^{\widetilde D_{n+2}}(q)=\displaystyle[2]_q[3]_{-q}[4]_q\cdot\left(\prod_{k=3}^n[2k+1]_q-q^k\right)\cdot\left(\prod_{k=1}^n\frac{1+q^{k+1}}{1-q^{n+k+2}}\right)$

%$\dfrac{[2]_q[3]_{-q}[4]_q\prod_{k=1}^n(1+q^{k+1})\prod_{k=3}^{n}([2k+1]_q-q^k)}{(1-q^{n+3})(1-q^{n+4})\cdots(1-q^{2n+2})}$\\
%$U_{\tilde C_{n}}^{\tilde D_{n+2}}(q)=\dfrac{U_{B_n}^{D_{n+1}}(q)\cdot\prod_{k=1}^n(1+q^{k+1})}{(1-q^{n+3})(1-q^{n+4})\cdots(1-q^{2n+2})}=U_{B_n}^{D_{n+1}}(q)\dfrac{(-q^2;q)_n}{(q^{n+3};q)_n}$

\item $U_{\widetilde C_{n}}^{\widetilde C_{2n+1}}(q)=\displaystyle\left(\prod_{k=1}^{2n+1}[k]_{(-1)^kq}\right)\cdot\left(\prod_{k=1}^n\frac{1+q^{2k}}{1-q^{2(n+k)+1}}\right)$

%$\dfrac{[2]_{q}[3]_{-q}\cdots[2n-1]_{-q}[2n]_{q}[2n+1]_{-q}\prod_{k=1}^{n}(1+q^{2k})}{(1-%q^{2n+3})(1-q^{2n+5})\cdots(1-q^{4n+1})}$\\
%$U_{\tilde C_{n}}^{\tilde C_{2n+1}}(q)=\dfrac{U_{B_n}^{A_{2n}}(q)\cdot\prod_{k=1}^{n}%(1+q^{2k})}{(1-q^{2n+3})(1-q^{2n+5})\cdots(1-q^{4n+1})}=U_{B_n}^{A_{2n}}(q)\dfrac{(-%q^2;q^2)_n}{(q^{2n+3};q^2)_n}$

\item $U_{\widetilde C_{n}}^{\widetilde C_{2n}}(q)=\displaystyle\left(\prod_{k=1}^{2n}[k]_{(-1)^kq}\right)\cdot\left(\prod_{k=1}^n\frac{1+q^{2k}}{1-q^{2(n+k)-1}}\right)$

%$\dfrac{[2]_{q}[3]_{-q}\cdots[2n-1]_{-q}[2n]_{q}\prod_{k=1}^{n}(1+q^{2k})}{(1-q^{2n+1})(1-q^{2n+3})\cdots(1-q^{4n-1})}$\\
%$U_{\tilde C_{n}}^{\tilde C_{2n}}(q)=\dfrac{U_{B_n}^{A_{2n-1}}(q)\cdot\prod_{k=1}^{n}(1+q^{2k})}{(1-q^{2n+1})(1-q^{2n+3})\cdots(1-q^{4n-1})}=U_{B_n}^{A_{2n-1}}(q)\dfrac{(-q^2;q^2)_n}{(q^{2n+1};q^2)_n}$
    \end{enumerate}
    \label{thm:UnfoldingSeriesAffineTypeC}
\end{theorem}

Similar to Theorem \ref{thm:affine_type_B}, each folding pair in Theorem \ref{thm:affine_type_C} has an associated folding pair of finite type whose unfolding polynomial appears as a factor of the corresponding affine unfolding series.\\

Unfolding polynomials studied in this paper have a close connection with other distribution functions for Coxeter groups.  In \cite{reiner1995distribution}, Reiner studies the $q$-Eulerian distribution for general Coxeter groups.  As part of his work, he calculates a certain distribution involving the frequency certain generators appearing in reduced words (see Proposition \ref{prop:Reiner}).  We show in Section \ref{S:Affine_foldings}, that some unfolding polynomials can be realized a specialization of these distributions (in fact, they play a critical role in the proof of Theorems \ref{thm:affine_type_B} and \ref{thm:affine_type_C}). \\

We remark that admissible set partitions and folding subgroups extend beyond the cases of finite and affine Coxeter groups.  As far as the authors know, unfolding polynomials are largely unknown for general folding subgroups. Several examples of admissible partitions beyond the finite and affine types are showcased by Elias and Heng in \cite{elias2024coxeter}.

\subsection*{Acknowledgements} We would like to thank Kirill Zainoulline for helpful discussions on Coxeter foldings.  The authors are supported by a grant from the Simons Foundation 941273.

\input{affinefoldingtable2}

\section{Preliminaries}

Let $W$ be a \textbf{Coxeter group} with simple generating set $S$ of rank $n$. Specifically, $W$ is generated by $S=\{s_1,\ldots, s_n\}$ with relations of the form
$$(s_is_j)^{m(s_i,s_j)}=e$$ for some value $m(s_i,s_j)\in\mathbb{Z}^+\cup\lbrace\infty\rbrace$ where $m(s_i,s_j) = 1$ if and only if $i=j$.  Each Coxeter group $W$ is equipped with a \textbf{length function} $\ell:W\to\mathbb{Z}_{\geq 0}$ where $\ell(w)$ denotes the minimum number of generators in $S$ needed to express $w$. Such a minimum-length expression is called a \textbf{reduced expression} for $w$.  The \textbf{Bruhat partial order} $\leq$ on $W$ is defined as follows: For $v,w\in W$,  we define $v\leq w$ if every reduced expression for $w$ has a subword that is a reduced expression for $v$. As a consequence, $v<w$ implies $\ell(v)<\ell(w)$ and hence the Bruhat order respects length. \\

% and we say that $(W,S)$ is a Coxeter system. 

%This system can be represented by a Coxeter \textit{graph} $\Gamma$ whose node %set is $S$ and whose edges are the unordered pairs $\lbrace s_i,s_j\rbrace$ %such that $m(s_i,s_j)\geq 3$. The edges with $m(s_i,s_j)\geq 4$ are labeled by %that number. We say that the system is irreducible if its Coxeter graph is %connected.\\

%Also, a Coxeter graph is simply-laced if $m(s_is_j)=2\text{ or } 3$ for all $s_i,s_j\in S$ and multiply-laced if $m(s_is_j)\geq 4$ for some $s_i,s_j\in S$.\\

%A \textit{subword} of a word $w=s_{i_1}s_{i_2}\cdots s_{i_{\ell(w)}}$ is a word %$v$ of the form $v=s_{j_1}s_{j_2}\cdots s_{j_{l(v)}}$, where %$j_1,j_2,\ldots,j_{\ell(v)}$ is a subsequence of $i_1,i_2,\ldots,i_{\ell(w)}$.\\

%A covering in Bruhat order, denoted $v\vartriangleleft w$, means that $v < w$ and $\ell(w)=\ell(v) + 1 $. So, the Bruhat order is a graded poset whose rank function is the length function and the Bruhat graph of $W$ is the directed graph whose nodes are the elements of $W$ and whose edges are given by the coverings $\vartriangleleft$.\\

For any $J\subseteq S$, let $W_J=\langle J \rangle$ denote the corresponding parabolic subgroup of $W$ generated by $J$.  Let $W^J$ denote the set of minimal length coset representatives of $W/W_J$.  Equivalently, we have
$$W^J= \lbrace w \in W : ws > w \text{ for all } s\in J\rbrace.$$

%We say that $w\in W^J$ if and only if no reduced expression for $w$ ends with a %letter from $J$. $W^J$ is also regarded as the set of $w\in W$ such that $D_R(w)\cap J=\emptyset$, where $D_R(w)$ is the right-descent set of $w$.\\

Every $w\in W$ has a unique factorization $w=w^J\cdot w_J$ such that $w^J\in W^J$ and $w_J\in W_J$ and $\ell(w) = \ell(w^J) + \ell(w_J)$ (see \cite{bjorner2005combinatorics}).  This factorization is called the \textbf{parabolic decomposition of $w$} with respect to $J$.  
The fact that parabolic decompositions are length preserving implies the following proposition.\\

%The \textbf{Poincar\'e polynomial} of $K\subseteq W$, denoted $K(q)$, is %defined as 
%$$K(q):=\displaystyle\sum_{w\in K}q^{\ell(w)}.$$  

\begin{proposition}\label{prop:parabolic_factorization}
    For any $J\subseteq S$, we have 
    $$\displaystyle\sum_{w\in W}q^{\ell(w)}=\left(\displaystyle\sum_{v\in W^J}q^{\ell(v)}\right)\cdot \left(\sum_{u\in W_J}q^{\ell(u)}\right).$$     
\end{proposition}

%\noindent A \textit{parabolic subgroup} of a Coxeter system $(W,S)$, denoted $W_J$, is a subgroup of $W$ generated by the set $J\subseteq S$. The system $(W_J, J)$ is also a Coxeter group. If $\ell_J$ is the length function on the subgroup $W_J$, then $\ell_J (w) = \ell(w)$ for all $w\in W_J$. The Coxeter graph for $(W_J, J)$ is obtained by removing all nodes in $S-J$ and their incident edges from the graph for $(W, S)$. If $W_J$ is finite, then it has a unique maximal element denoted by $w_0(J)$.\\

%\noindent Define the set $W^J = \lbrace w \in W : ws > w \text{ for all } s\in J\rbrace$. We say that $w\in W^J$ if and only if no reduced expression for $w$ ends with a letter from $J$. Equivalently, defining the \textit{right descents} set of $w\in W$ as $D_R(w)=\lbrace s\in S:ws<w\rbrace$, $W^J$ is the set of $w\in W$ such that $D_R(w)\cap J=\emptyset$. 

%The set $W^J$ provides a factorization of $W$, that is, every $w\in W$ has a unique factorization $w=w^J\cdot w_J$ such that $w^J\in W^J$ and $w_J\in W_J$ and $\ell(w) = \ell(w^J) + \ell(w_J)$. As a consequence, each left coset $wW_J$ has a unique representative of minimal length. The system of such minimal coset representatives is $W^J$. In particular, we have\\

\subsection{Folding subgroups and parabolic decompositions}

Let $\widehat W$ be a folding subgroup of $W$ corresponding to an admissible set partition $\widehat S$ of $S$.  The corresponding partition function gives a map $\pi:S\rightarrow R$ where $R$ denotes the generators of $\widehat W$.
In particular, we have that $\pi(s)=r$ if and only if $s\leq \phi(r)$ where $\phi:\widehat W\rightarrow W$ denotes the embedding homomorphism.  The following properties are proved by M\"uhlherr in \cite{muhlherr1993coxeter} and were later (independently) proved in  ~\cite{elias2024coxeter}.\\

\begin{lemma}\label{lem:folding_length}
Let $\widehat W$ be a folding subgroup with the notation above and let $w=r_{i_1}\cdots r_{i_k}\in \widehat W$ be a reduced expression in the generators of $R$.  Then the following are true:
\begin{enumerate}
\item $\displaystyle \ell(\phi(w))=\sum_{j=1}^k \ell(\phi(r_{i_j})).$
\item Let $s\in \pi^{-1}(r)$.  Then $w\cdot r<w$ in $\widehat W$ if and only if $\phi(w)\cdot s<\phi(w)$ in $W$.  
\end{enumerate}
\end{lemma}

The primary tool we use to prove Theorem \ref{thm:MainTheorem} is to show that the unfolding polynomial respects the factorization of the Poincar\'e polynomial in Proposition \ref{prop:parabolic_factorization}.  \\

%Using  Lemma \ref{lem:folding_length}, we prove the following proposition.\\

\begin{proposition}\label{prop:folding_parabolic}
Let $\widehat{J}\subseteq R$ and define 
$$J:=\lbrace s\in\pi^{-1}(r):r\in\widehat{J}\rbrace\subseteq S.$$ 
Consider the parabolic subgroups $\widehat{W}_{\widehat{J}}$ and $W_J$ of $\widehat W$ and $W$, respectively. Then the following are true:
\begin{enumerate}
    \item If $w\in\widehat{W}_{\widehat{J}}$, then $\phi(w)\in W_J$.
    \label{ParabolicSubgroupsUnfoldingProposition}
    \item If $w\in\widehat{W}^{\widehat{J}}$, then $\phi(w)\in W^J$.
    \label{ParabolicQuotientsUnfoldingProposition}
\end{enumerate}    
\end{proposition}

\begin{proof}
Part 1 follows immediately from the definition of $J$.  To prove part 2, let $w\in\widehat{W}^{\widehat{J}}$ and suppose $r$ is right descent of $w$ (i.e. $w\cdot r<w$).  Then $r\notin\widehat{J}$ and, by Lemma \ref{lem:folding_length}, $\phi(w)\cdot s<\phi(w)$ for all $s\in\pi^{-1}(r)$. Since $r\notin\tilde{J}$, we have $s\notin J$ and thus $\phi(w)\in W^J$.

%     Let $w=r_1r_2\cdots r_k$ be a reduced expression for $w\in \tilde{W}_{\tilde{J}}$. Then $\phi(w)=\phi(r_1)\phi(r_2)\cdots\phi(r_k)$ where $\phi(r_i)$ is a product of $S-$generators for each $i$ and $s\leq \phi(r_i)$ for each of those generators. Then $\pi(s)=r_i$ and $s\in\pi^{-1}(r_i)$. Thus, $s\in J$ and $\phi(w)\in W_J$, proving part 1.\\

%     To prove part 2,  let $w\in\tilde{W}^{\tilde{J}}$ be reduced. Then $D_R(w)\cap\tilde{J}=\emptyset$. Suppose $r\in D_R(w)$. Then $r\notin\tilde{J}$ and $wr<w$. By Lemma \ref{lem:folding_length}, $\phi(w)s<\phi(w)$ for all $s\in\pi^{-1}(r)$. Then $s\in D_R(\phi(w))$. Since $r\notin\tilde{J}$, then $s\notin J$. So $D_R(\phi(w))\cap J=\emptyset$. Thus $\phi(w)\in W^J$.\\
     
\end{proof}

One consequence of Proposition \ref{prop:folding_parabolic} is that if $w=vu$ is a parabolic decomposition with respect to $\widehat J$ in $\widehat W$, then $\phi(w)=\phi(v)\phi(u)$ is a parabolic decomposition with respect to $J$ in $W$.  For any $K\subseteq \widehat W$, we define the unfolding series of $K$ in $W$ as 
$$U_{K}^W(q):=\sum_{w\in K}q^{\ell(\phi(w))}.$$

\begin{coro}
\label{cor:ParabolicFactorizationFolding}
Let $\widehat W$ be a folding subgroup of $W$ with $\widehat J\subset R$ and $J\subset S$ as in Proposition \ref{prop:folding_parabolic}.  Then 
\begin{equation*}U_{\widehat{W}}^W(q)=U_{\widehat{W}^{\widehat{J}}}^W(q)\cdot U_{\widehat{W}_{\widehat{J}}}^{W_J}(q)
    %\label{UnfoldingPolynomialFactorization}
    \end{equation*}
\end{coro}

\begin{proof}
The corollary follows from Propositions \ref{prop:parabolic_factorization} and \ref{prop:folding_parabolic}.
\end{proof}

\section{Coxeter foldings of finite type}

In this section, we present a proof of Theorem \ref{thm:MainTheorem} by inducting on the formula found in Corollary \ref{cor:ParabolicFactorizationFolding} through a careful choice of $\widehat J\subset R$.   For parts (1)-(3) in Theorem \ref{thm:MainTheorem}, we find an explicit formula for $U_{\widehat{W}^{\widehat{J}}}^W(q)$ and observe that the second factor $U_{\widehat{W}_{\widehat{J}}}^{W_J}(q)$ belongs to a lower rank case of the same calculation.  For parts (4) and (5) of Theorem \ref{thm:MainTheorem}, we find explicit formulas for both $U_{\widehat{W}^{\widehat{J}}}^W(q)$ and $U_{\widehat{W}_{\widehat{J}}}^W(q)$.

\subsection{Classical foldings}

Since the folding subgroup $\widehat{W}$ is of type $B_n$ for each classical folding, we start by describing the parabolic subgroup $\widehat{W}_{\widehat{J}}$ and the respective system of minimal coset representatives $\widehat{W}^{\widehat{J}}$. Then we will describe their respective unfolding in the group $W$ for each type: $A_{2n-1}$, $A_{2n}$, and $D_{n+1}$.\\

%\noindent For the unfolding of $B_n$, the factorization of %$U_{\tilde{W}}^W(q)$ in Corollary \ref{ParabolicFactorizationFolding} %leads to a recursive formula. 

%We  and then use induction to find formulas 1, 2, and 3 in Theorem \ref{thm:MainTheorem}.\\

%\noindent 

%\subsection{Unfolding of $B_n$}

Write $\widehat{W}=B_n$ with generating set $R=\lbrace r_1,\ldots,r_n\rbrace$. The Coxeter relations on $B_n$ are given by 
$$(r_ir_{i+1})^3= (r_{n-1}r_n)^4 =e\quad \text{for $1\leq i<n-1$}\quad\text{and}\quad (r_ir_j)^2=e\quad\text{otherwise.}$$

Let $\widehat{J}=R\setminus\lbrace r_1\rbrace$. Then the parabolic subgroup $\widehat{W}_{\widehat{J}}$ of $\widehat{W}=B_n$ generated by $\widehat{J}$ is isomorphic to a Coxeter group of type $B_{n-1}$.  We write $\widehat{W}_{\widehat{J}}=B_{n-1}$.  The next lemma gives an explicit description of $\widehat{W}^{\widehat{J}}$.  We leave the proof as a simple exercise.

%The set of minimal coset representatives $$\tilde{W}^{\tilde{J}}=\{w\in %B_n\ |\ wr_1>w\}.$$ 

%The elements in $B_{n-1}$ can be unfolded as elements in $A_{2n-3}$, $A_{2n-2}$, or $D_n$ in accordance to Table \ref{DescriptionFiniteFoldingTable}.

%will be the set of words $u$ in $\tilde{W}$ such that no reduced word %for $u$ ends in an element of $\tilde{J}$, or equivalently,  the set of %words $u$ in $\tilde{W}$ such that every reduced word for $u$ ends in %$r_1$.\\

\begin{lemma}\label{lem:ElementsQuotientTypeB}
    Let $u\in \widehat{W}^{\widehat{J}}$. If $u\neq e$, then $u$ is one of the following: 
    \begin{itemize}
        \item $ r_ir_{i-1}\cdots r_2r_1$, for some $1\leq i\leq n$
        \item $r_ir_{i+1}\cdots r_{n-1}r_nr_{n-1}\cdots r_2r_1$, for some $1\leq i<n$.\\
        
    \end{itemize}
    
   % $ r_ir_{i-1}\cdots r_3r_2r_1$, $i\in [n]$, or (3) $r_ir_{i+1}\cdots %r_{n-1}r_nr_{n-1}\cdots r_3r_2r_1$, $i\in [n-1]$.
\end{lemma}

\begin{proof}[Proof of Thoerem \ref{thm:MainTheorem} parts (1)-(3)]

%In the next three sections, we describe the unfolding of $\tilde W^{\tilde J}$ in each %of the cases of $W=A_{2n-1}, A_{2n}$, and $D_{n+1}$.

%\subsubsubsection{The unfoldings of $B_n$ in $A_{2n-1}$ and $A_{2n}$}

For part (1), let $W=A_{2n-1}$ with generating set $S=\lbrace s_1,\ldots,s_{2n-1}\rbrace$ and consider the folded embedding $\phi:B_n\rightarrow A_{2n-1}$ given by $$\phi(r_i)=s_is_{2n-i}\ \text{for $i<n$}\quad\text{and}\quad  \phi(r_n)=s_n.$$ Since $\pi^{-1}(r_1)=\{s_1,s_{2n-1}\}$, we let $J=S\setminus\lbrace s_1,s_{2n-1}\rbrace$ as in Proposition \ref{prop:folding_parabolic}.  The parabolic subgroup $W_{J}$ is isomorphic to $A_{2n-3}$ and Corollary \ref{cor:ParabolicFactorizationFolding} says that
\begin{equation*}
    U_{B_n}^{A_{2n-1}}(q)=U_{\widehat{W}^{\widehat{J}}}^{A_{2n-1}}(q)\cdot U_{B_{n-1}}^{A_{2n-3}}(q).
    %\label{UnfoldingFactorizationAfromBodd}
\end{equation*}
Lemma \ref{lem:ElementsQuotientTypeB} implies
$$U_{\widehat{W}^{\widehat{J}}}^{A_{2n-1}}(q)=\displaystyle\sum_{i=0}^{n-1}q^{2i}+\sum_{i=n+1}^{2n}q^{2i-3}=[2n-1]_{-q}[2n]_q.$$
By induction, we have
\begin{equation*}
    U_{B_n}^{A_{2n-1}}(q)=[2]_q[3]_{-q}[4]_q[5]_{-q}\cdots [2n-1]_{-q}[2n]_q%\label{MainTheoremTypeAodd}
\end{equation*}
which proves part (1).

%generated by $J$ is isomorphic to a Coxeter group of type $, and write %$W_J=A_{2n-3}$. 

%The system of minimal coset representatives $W^{J}$ will be the set of words $w$ in $A_{2n-%}$ such that no reduced word for $w$ ends in an element of $J$, or equivalently,  the set %of words $w$ in $A_{2n-1}$ such that every reduced word for $w$ ends in $s_1$ or $s_{2n-%1}$.\\

%For $u\in\tilde{W}^{\tilde{J}}$ as in Lemma \ref{lem:ElementsQuotientTypeB}, then %$\phi(u)\in W^{J}$ and by 

%\subsection{The unfolding of $B_n$ to $A_{2n}$}

For part (2), let $W=A_{2n}$ and consider the folded embedding $\phi:B_n\rightarrow A_{2n}$ given by 
$$\phi(r_i)=s_is_{2n-i+1}\ \text{for $i<n$}\quad\text{and}\quad  \phi(r_n)=s_ns_{n+1}s_n.$$
Since $\pi^{-1}(r_1)=\{s_1,s_{2n}\}$, we let $J=S\setminus\lbrace s_1,s_{2n}\rbrace$ as in Proposition \ref{prop:folding_parabolic}.  The parabolic subgroup $W_{J}$ is isomorphic to $A_{2n-2}$ and Corollary \ref{cor:ParabolicFactorizationFolding} implies 
  \begin{equation*}
      U_{B_n}^{A_{2n}}(q)=U_{\widehat{W}^{\widehat{J}}}^{A_{2n}}(q)\cdot U_{B_{n-1}}^{A_{2n-2}}(q).
    %\label{UnfoldingFactorizationAfromBeven}
  \end{equation*}
  Using Lemma \ref{lem:ElementsQuotientTypeB}, we get
  $$U_{\widehat{W}^{\widehat{J}}}^{A_{2n}}(q)=\displaystyle\sum_{i=0}^{n-1}q^{2i}+\sum_{i=n}^{2n-1}q^{2i+1}=[2n]_{q}[2n+1]_{-q}.$$
  Again by induction, we have 
$$U_{B_n}^{A_{2n}}(q)=[2]_q[3]_{-q}[4]_q[5]_{-q}\cdots [2n-1]_{-q}[2n]_q[2n+1]_{-q}.$$
This completes the proof of part (2).  
%Take $J=S\setminus\lbrace s_1,s_{2n}\rbrace$. Then the parabolic subgroup $W_{J}$ of $W$ generated by $J$ is isomorphic to a Coxeter group of type $A_{2n-2}$ and write $W_J=A_{2n-2}$. The system of minimal coset %representatives $W^{J}$ will be the set of words $w$ in $A_{2n}$ such that no reduced word for $w$ ends in an element of $J$, or equivalently,  the set of words $w$ in $A_{2n}$ such that every reduced word for $w$ ends in %$s_1$ or $s_{2n}$.\\

 % For $u\in\tilde{W}^{\tilde{J}}$ as in Lemma \ref{lem:ElementsQuotientTypeB}, then $\phi(u)\in W^{J}$ and by Corollary \ref{cor:ParabolicFactorizationFolding}
%  \begin{equation*}
%      U_{B_n}^{A_{2n}}(q)=U_{\tilde{W}^{\tilde{J}}}^{A_{2n}}(q)U_{B_{n-1}}^{A_{2n-2}}(q)
%    %\label{UnfoldingFactorizationAfromBeven}
%  \end{equation*}

%where $$U_{\tilde{W}^{\tilde{J}}}^{A_{2n}}(q)=\displaystyle\sum_{i=0}^{n-1}q^{2i}+\sum_{i=n}^{2n-1}q^{2i+1}=[2n]_{q}[2n+1]_{-q}$$
%and by induction
%$$U_{B_n}^{A_{2n}}(q)=[2]_q[3]_{-q}[4]_q[5]_{-q}\cdots [2n-1]_{-q}[2n]_q[2n+1]_{-q}$$

%From the polynomials $U_{B_n}^{A_{2n-1}}(q)$ and $U_{B_n}^{A_{2n}}(q)$, we have\\

%\begin{cor} Consider the folding $A_k$ to $B_{\lfloor(k+1)/2\rfloor}$. Then
%\begin{eqnarray*}
%    \sum_{(v,w)\in B_{\lfloor(k+1)/2\rfloor}\times B_{\lfloor(k+1)/2\rfloor}}(-1)^{\ell_B(w)}q^{\ell_B(w)+\ell_A(\phi(v))} & = & \sum_{(x,y)\in A_{k}\times B_{\lfloor(k+1)/2\rfloor}}(-1)^{\ell_A(x)}q^{\ell_A(x)+\ell_B(y)}
%\end{eqnarray*}
%    \label{CorollaryMainTheoremTypeA}
%\end{cor}
    
%\subsubsubsection{The unfolding of $B_n$ in $D_{n+1}$}

For part (3) of Theorem \ref{thm:MainTheorem}, write $W=D_{n+1}$ with generators $S=\lbrace s_1,\ldots,s_{n+1}\rbrace$ and consider the folded embedding $\phi:B_n\rightarrow D_{n+1}$ given by 
$$\phi(r_i)=s_i\ \text{for $i<n$}\quad\text{and}\quad  \phi(r_n)=s_ns_{n+1}.$$  First note that if $n=2$, then the folded embedding $\phi:B_2\rightarrow D_3$ is equivalent to $\phi:B_2\rightarrow A_3$ as in part (1) and hence
$$U_{B_2}^{D_3}(q)=U_{B_2}^{A_3}(q)=[2]_q[3]_{-q}[4].$$
Now assume $n\geq 3$.  Since $\pi^{-1}(r_1)=\{s_1\}$, we let $J=S\setminus\lbrace s_1\rbrace$ as in Proposition \ref{prop:folding_parabolic}. Then the subgroup $W_{J}$ of $D_{n+1}$ generated by $J$ is isomorphic to $D_{n}$ and write $W_J=D_{n}$.  Corollary \ref{cor:ParabolicFactorizationFolding} says that
\begin{equation*}
    U_{B_n}^{D_{n+1}}(q)=U_{\widehat{W}^{\widehat{J}}}^{D_{n+1}}(q)\cdot U_{B_{n-1}}^{D_n}(q)
    %\label{UnfoldingFactorizationD}
\end{equation*}
and Lemma \ref{lem:ElementsQuotientTypeB} yields
$$U_{\widehat{W}^{\widehat{J}}}^{D_{n+1}}(q)=\displaystyle\sum_{\substack{0\leq i\leq 2n \\ i\neq n}}q^{i}=[2n+1]_{q}-q^n.$$
By induction, we have
\begin{equation*}
    U_{B_n}^{D_{n+1}}(q)=[2]_q[3]_{-q}[4]\prod_{k=3}^{n}\left([2k+1]_q-q^k\right).
    %\label{MainTheoremTypeD}
\end{equation*}

%The system of minimal coset representatives $W^{J}$ will be the set of words $w$ in $D_{n+1}$ such that no reduced word for $w$ ends in an element of $J$, or equivalently,  the set of words $w$ in $D_{n+1}$ such that every %reduced word for $w$ ends in $s_1$.\\

%For $u\in\tilde{W}^{\tilde{J}}$ as in Lemma \ref{LemmaElementsQuotientTypeB}, then $\phi(u)\in W^{J}$ and by Corollary \ref{ParabolicFactorizationFolding}

\end{proof}

With the formulas for $U_{B_n}^{A_{2n-1}}(q)$ and $U_{B_n}^{A_{2n}}(q)$, we can prove Corollary \ref{cor:length_identity}.  Note that for the full rank function on the groups $A_n$ and $B_n$, we have 
$$A_n(q):=\sum_{w\in A_n} q^{\ell_A(w)}=[1]_q[2]_q\cdots [n]_q\quad\text{and}
\quad B_n(q):=\sum_{w\in B_n} q^{\ell_B(w)}=[2]_q[4]_q\cdots [2n]_q.$$
Corollary \ref{cor:length_identity} follows from the identities
$$B_n(-q)\cdot U_{B_n}^{A_{2n-1}}(q) = A_{2n-1}(-q)\cdot B_n(q)\quad\text{and}
\quad B_n(-q)\cdot U_{B_n}^{A_{2n}}(q) = A_{2n}(-q)\cdot B_n(q).$$

\subsection{Lusztig foldings}

The dihedral group $\widehat W=I_2(n+1)$ is generated by $R=\lbrace r_1,r_2\rbrace$ subject to the Coxeter relations $$r_1^2=r_2^2=(r_1r_2)^{n+1}=e.$$  Consider the embedding $\phi:I_2(n+1)\to A_n$ given by 
$$\displaystyle\phi(r_1)=\prod_{\substack{s_i\in S\\i\text{ odd}}}s_{i}\quad\text{and}\quad
    \displaystyle\phi(r_2)=\prod_{\substack{s_i\in S\\i\text{ even}}}s_i$$ 
    where $S=\{s_1,\ldots,s_n\}$ denote the generators of $A_n$.
\begin{proof}[Proof of Thoerem \ref{thm:MainTheorem} parts (4) and (5)]
Let $\widehat{J}=\lbrace r_2\rbrace$.  Then the parabolic subgroup $\widehat{W}_{\widehat{J}}=\{e,r_2\}$ and the set of minimal length coset representatives 
$\widehat{W}^{\widehat{J}}$ consists of the identity element and all expressions $\underset{j}{\underbrace{\ldots r_2r_1r_2r_1}}$ for some $1\leq j\leq n$.  In particular, $\widehat{W}^{\widehat{J}}$ contains a unique element of length $j$ for every $0\leq j\leq n$. 
 Let $J=\lbrace s_i\ |\ i\text{ is even}\rbrace$.  Then by Corollary \ref{cor:ParabolicFactorizationFolding},
  \begin{equation*}
      U_{I_2(n+1)}^{A_n}(q)=U_{\widehat{W}^{\widehat{J}}}^{A_{n}}(q)\cdot U_{\widehat W_{\widehat J}}^{W_J}(q).
    %\label{UnfoldingFactorizationAfromI}
  \end{equation*}
Suppose $n$ is even and write $n=2m$. Then $\ell_A(\phi(r_1))=\ell_A(\phi(r_2))=m$ and hence
$$U_{\tilde W_{\tilde J}}^{W_J}(q)=1+q^m=[2]_{q^m}$$
and 
$$U_{\widehat{W}^{\widehat{J}}}^{A_{n}}(q)=1+q^m+q^{2m}+\cdots + q^{nm}=[n+1]_{q^m}.$$
Otherwise, suppose $n$ is odd and write $n=2m-1$.  Then $\ell_A(\phi(r_1))=m$ and $\ell_A(\phi(r_2))=m-1$.  This implies 
$$U_{\widehat W_{\widehat J}}^{W_J}(q)=1+q^{m-1}=[2]_{q^{m-1}}$$
and 
$$U_{\widehat{W}^{\widehat{J}}}^{A_{n}}(q)=(1+q^{m})(1+q^{n}+q^{2n}+\cdots+q^{(m-1)n})=[2]_{q^m}\cdot[m]_{q^n}.$$
These calculations imply parts (4) and part (5) of Theorem \ref{thm:MainTheorem} which completes the proof.
\end{proof}

\section{Coxeter foldings of affine type}\label{S:Affine_foldings}

In this section we prove Theorems \ref{thm:affine_type_A}, \ref{thm:affine_type_B}, and \ref{thm:affine_type_C} on the unfolding series for irreducible Coxeter groups of affine type.  Unlike the finite case, we do not apply Corollary \ref{cor:ParabolicFactorizationFolding} in our calculations (although  Corollary \ref{cor:ParabolicFactorizationFolding} holds for affine foldings).  We first show that Theorem \ref{thm:affine_type_A} follows from a direct calculation based on the Poincar\'{e} series of affine type A.  For Theorems \ref{thm:affine_type_B} and \ref{thm:affine_type_C} we apply a certain distribution formula described by Reiner in \cite{reiner1995distribution} (see Proposition \ref{prop:Reiner}).

%Though Corollary \ref{cor:ParabolicFactorizationFolding} does not require that the Coxeter group and its respective folding subgroup are finite, its application on affine Coxeter groups is very challenging. For the finite case, the choice of the subset $\tilde J$ of $R$ produced a set of minimal coset representatives $\tilde W^{\tilde J}$ following an easily recognizable pattern. So, the polynomials $U_{\tilde{W}^{\tilde{J}}}^W(q)$ and $ U_{\tilde{W}_{\tilde{J}}}^{W_J}(q)$ can both be computed as shown in Lemma \ref{lem:ElementsQuotientTypeB} and Lemma \ref{LemmaElementsQuotientTypeI}. However, for affine Coxeter groups, the elements of such a set of minimal coset representatives won't follow a recognizable pattern for any subset $\tilde J$.\\

%The computation of the unfolding series for each embedding presented in Table \ref{DescriptionAffineFoldingTable1} and Table \ref{DescriptionAffineFoldingTable2} follows as a specialization of the Poincar\'e series of $\tilde A_{n-1}$ and the Reiner's distribution functions $\tilde B_n(a,q)$ and $\tilde C_n(a,b,q)$ described in \cite{reiner1995distribution}. 

\subsection{The unfolding of $\widetilde A_{n-1}$}
For any $m\geq 1$, we consider the embedding $\phi:\widetilde A_{n-1}\rightarrow \widetilde{A}_{mn-1}$ given by $$\phi(r_i)=\prod_{j=0}^{m-1} s_{i+jn}$$
where $R=\{r_0,\ldots, r_{n-1}\}$ and $S=\{s_0,\ldots, s_{mn-1}\}$ denote the generators of $A_{n-1}$ and $\widetilde{A}_{mn-1}$ respectively (see Table \ref{tab:DescriptionAffineFoldingTableTypeA}).  We also let $\ell$ and $\ell_{\widetilde A}$ represent the length functions on $\widetilde A_{kn-1}$ and $\widetilde A_{n-1}$, respectively. In \cite{Bott1979}, Bott shows that the Poincar\'e series of the group $\widetilde A_{n-1}$ is given by
\begin{equation}
    \widetilde A_{n-1}(q)=\sum_{w\in\widetilde A_{n-1}}q^{\ell_{\widetilde A}(w)}=\displaystyle\prod_{k=1}^n\frac{[k]_{q}}{1-q^{(k-1)}}
    \label{PoincareSeriesAffineA}
\end{equation}

Note that $\ell(\phi(r))=m\ell_{\widetilde A}(r)$ for any $r\in R$ and therefore Lemma \ref{lem:folding_length} implies $\ell(\phi(w))=m\ell_{\widetilde A_{n-1}}(w)$ for any $w\in\widetilde A_{n-1}$. Substituting $q\mapsto q^m$ in Equation \eqref{PoincareSeriesAffineA} gives
$$U_{\widetilde A_{n-1}}^{A_{mn-1}}(q)= \sum_{w\in\widetilde A_{n-1}}q^{\ell(\phi(w))} = \sum_{w\in\widetilde A_{n-1}}q^{m\ell_{\widetilde A}(w)}=\widetilde A_{n-1}(q^m)=\displaystyle\prod_{k=1}^n\frac{[k]_{q^m}}{1-q^{(k-1)m}}.$$
This proves Theorem \ref{thm:affine_type_A}.

\subsection{The unfoldings of $\widetilde B_n$ and $\widetilde C_n$}
In this section we consider the unfoldings of affine Coxeter groups $\widetilde B_n$ and $\widetilde C_n$ in parts (ii) and (iii) of Table \ref{AffineFoldingTable} (see also Tables \ref{DescriptionAffineFoldingTable1} and \ref{DescriptionAffineFoldingTable2}). The Coxeter graphs of $\widetilde B_n$ and $\widetilde C_n$ are given by: 
$$\widetilde{B}_n:\quad 
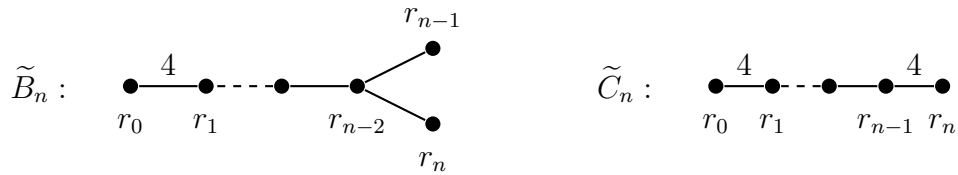
\begin{tikzpicture}[
      thick,
      acteur/.style={
        circle,
        fill=black,
        thick,
        inner sep=2pt,
        minimum size=0.2cm        
      },
      baseline={(0,1.85)}    
    ] 
      \node (a1) at (1,2) [acteur,label=below:$r_{n-2}$]{};
      \node (a2) at (0,2)[acteur]{}; 
      \node (a3) at (-1,2) [acteur,label=below:$r_1$]{}; 
      \node (a4) at (-2,2) [acteur,label=below:$r_0$]{}; 
      \node (a5) at (2,2.5) [acteur,label=above:$r_{n-1}$]{}; 
      \node (a6) at (2,1.5) [acteur,label=below:$r_{n}$]{};
      
      \draw (a1) -- (a2); 
      \draw[dashed] (a2) -- (a3); 
      \draw (a3) -- node[midway, above, sloped] {4} (a4);
      \draw (a1) -- (a5);
      \draw (a1) -- (a6);

      %\draw[blue, <->] (a5)--(a6);
%     \draw (a7) -- (a5);
%     \draw[red] (a9) -- (a5);
%     \draw[red] (a9) -- (a4);
%      \draw[red] (a5) -- (a3);
%      \draw[red] (a6) -- (a2);

    \end{tikzpicture}
    \qquad\qquad  \widetilde{C}_n:\quad 
\begin{tikzpicture}[
      thick,
      acteur/.style={
        circle,
        fill=black,
        thick,
        inner sep=2pt,
        minimum size=0.2cm
      },
      baseline={(0,1.85)} 
    ]       
\node (c1) at (7.25,2) [acteur,label=below:$r_{n-1}$]{};
      \node (c2) at (6.5,2)[acteur]{}; 
      \node (c3) at (5.75,2) [acteur,label=below:$r_1$]{}; 
      \node (c4) at (5,2) [acteur,label=below:$r_0$]{}; 
      \node (c5) at (8,2) [acteur,label=below:$r_n$]{}; 
%      \node (a6) at (2.5,0.5) [acteur,label=below:$s_{n+1}$]{};
%      \node (a7) at (1,0.5) [acteur]{}; 
%      \node (a8) at (-0.5,0.5) [acteur,label=below:$s_{2n-2}$]{}; 
%      \node (a9) at (-2,0.5) [acteur,label=below:$s_{2n-1}$]{}; 
      \draw (c1) -- (c2); 
      \draw[dashed] (c2) -- (c3); 
      \draw (c3) --node[midway, above, sloped] {4} (c4);
      \draw (c1) -- node[midway, above, sloped] {4} (c5);
    \end{tikzpicture}$$
Since the generators $r_0$ (resp. $r_0, r_n$) in type $\widetilde B_n$ (resp. $\widetilde C_n$) only involve braid relations of even length, the number of times these generators appears in a reduced word of a given element is unique.  For $w\in\widetilde B_n$, let $r_B(w)$ denote then number of times $r_0$ appears in a reduced word of $w$.  Similarly, for $w\in\widetilde C_n$, let $r_C(w)$ and $s_C(w)$ denote the number of times $r_0$ and $r_n$ appear in a reduced word of $w$ respectively.  The following formulas appear in \cite{reiner1995distribution} where Reiner studies $q$-Eulerian distribution functions on Coxeter groups.

\begin{proposition}[Reiner, 1992]\label{prop:Reiner}
Denote the $q$-Pochhammer symbol by
$$
(x;q)_n:=\prod_{k=0}^{n-1}(1-q^kx).
$$
Then the following formulas hold:
\begin{equation}
    \widetilde B_n(a,q):=\sum_{w\in\widetilde B_n}a^{r_B(w)}q^{\ell_{\widetilde B}(w)}=\dfrac{(-aq;q)_n(-q;q)_{n-1}[n]_q!}{(aq^n;q)_n}
    \label{PoincareSeriesAffineB}
\end{equation}
\item \begin{equation}
    \widetilde C_n(a,b,q):=\sum_{w\in\widetilde C_n}a^{r_C(w)}b^{s_C(w)}q^{\ell_{\widetilde C}(w)}=\dfrac{(-aq;q)_n(-bq;q)_{n}[n]_q!}{(abq^{n+1};q)_n}
    \label{PoincareSeriesAffineC}
\end{equation}
\end{proposition}

%Following Reiner's work, consider the Coxeter group $\widetilde B_n$. Let %$r(w)$ be the number of times the generator $r_0$ appears in any reduced %word for $w$, and let $\ell_{\widetilde B}(w)$ be its respective length. In \cite{reiner1995distribution}, Reiner shows that
%\begin{equation}
%    \widetilde B_n(a,q)=\sum_{w\in\widetilde %B_n}a^{r(w)}q^{\ell_{\widetilde B}(w)}=\dfrac{(-aq;q)_n(-q;q)_{n-1}[n]_q!}%{(aq^n;q)_n}
%    \label{PoincareSeriesAffineB}
%\end{equation}

We prove Theorems \ref{thm:affine_type_B} and \ref{thm:affine_type_C} by showing that each unfolding series is a specialization of the formulas in Proposition \ref{prop:Reiner}.

\begin{proof}[Proof of Theorem \ref{thm:affine_type_B}]

Theorem \ref{thm:affine_type_B} focuses on the three foldings of $\widetilde{B}_n$ in part (ii) of Table \ref{tab:affinegroupstable} which are explicitly described in Table \ref{DescriptionAffineFoldingTable1}.  We first consider the embedding $\phi:\widetilde B_n\to\widetilde D_{n+1}$ from part 1.  In this case, we can see that $\ell_{\widetilde B}(\phi(r_0))=2$ and $\ell_{\widetilde B}(\phi(r_i))=1$ for $i>0$. Lemma \ref{lem:folding_length} implies $\ell_{\widetilde D}(\phi(w))=\ell_{\widetilde B}(w)+r(w)$ for all $w\in\widetilde B_n$.
Substituting $a\mapsto q$ in Equation \eqref{PoincareSeriesAffineB} yields:
$$U_{\widetilde B_n}^{\widetilde D_{n+1}}(q)=\sum_{w\in\widetilde B_n}q^{\ell_{\widetilde D}(\phi(w))}=\sum_{w\in\widetilde B_n}q^{r(w)}q^{\ell_{\widetilde B}(w)}=\widetilde B_n(q,q)=\dfrac{(-q^2;q)_n(-q;q)_{n-1}[n]_q!}{(q^{n+1};q)_n}.$$

For part 2, we have embedding $\phi:\widetilde B_n\to\widetilde D_{2n}$. Here we have $\ell_{\widetilde B}(\phi(r_0))=1$ and $\ell_{\widetilde B}(\phi(r_i))=2$ for $i>0$.
Lemma \ref{lem:folding_length} implies $\ell_{\widetilde D}(\phi(w))=2\ell_{\widetilde B}(w)-r(w)$. Substituting $a\mapsto q^{-1}$ and $q\mapsto q^2$ in Equation \eqref{PoincareSeriesAffineB} gives
$$U_{\widetilde B_n}^{\widetilde D_{2n}}(q)=\sum_{w\in\widetilde B_n}q^{\ell_{\widetilde D}(\phi(w))}=\sum_{w\in\widetilde B_n}q^{-r(w)}q^{2\ell_{\widetilde B}(w)}=\widetilde B_n(q^{-1},q^2)=\dfrac{(-q;q^2)_n(-q^2;q^2)_{n-1}[n]_{q^2}!}{(q^{2n-1};q^2)_n}.$$

For part 3, we have the embedding $\phi:\widetilde B_n\to\widetilde D_{2n+1}$ with $\ell_{\widetilde B}(\phi(r_0))=3$ and $\ell_{\widetilde B}(\phi(r_i))=2$ for $i>0$.
In this case, we have $\ell_{\widetilde D}(\phi(w))=2\ell_{\widetilde B}(w)+r(w)$ and substituting $a\mapsto q$ and $q\mapsto q^2$ in Equation \eqref{PoincareSeriesAffineB} yields
$$U_{\widetilde B_n}^{\widetilde D_{2n+1}}(q)=\sum_{w\in\widetilde B_n}q^{\ell_{\widetilde D}(\phi(w))}=\sum_{w\in\widetilde B_n}q^{r(w)}q^{2\ell_{\widetilde B}(w)}=\widetilde B_n(q,q^2)=\dfrac{(-q^3;q^2)_n(-q^2;q^2)_{n-1}[n]_{q^2}!}{(q^{2n+1};q^2)_n}.$$
This completes the proof.
\end{proof}

\begin{proof}[Proof of Theorem \ref{thm:affine_type_C}]
The proof of Theorem \ref{thm:affine_type_C} follows the same arguments as in the proof of Theorem \ref{thm:affine_type_B}.  In this case, we apply Equation \eqref{PoincareSeriesAffineC} to each of the seven foldings of $\widetilde{C}_n$ given in Table \ref{DescriptionAffineFoldingTable2}.  Below, we list the corresponding relationships between length functions and leave the details of the calculations as an exercise.

\begin{enumerate}
\item $\ell_{\widetilde A}(\phi(w))=2\ell_{\widetilde C}(w)+r(w)+s(w)$ \ and \ $U_{\widetilde C_n}^{\widetilde A_{2n+1}}(q)=\widetilde C_n(q,q,q^2)$

\smallskip

\item $\ell_{\widetilde A}(\phi(w))=2\ell_{\widetilde C}(w)+r(w)-s(w)$ \ and \ $U_{\widetilde C_n}^{\widetilde A_{2n}}(q)=\widetilde C_n(q,q^{-1},q^2)$

\smallskip

\item $\ell_{\widetilde A}(\phi(w))=2\ell_{\widetilde C}(w)-r(w)-s(w)$ \ and \ $U_{\widetilde C_n}^{\widetilde A_{2n-1}}(q)=\widetilde C_n(q^{-1},q^{-1},q^2)$

\smallskip

\item $\ell_{\widetilde B}(\phi(w))=\ell_{\widetilde C}(w)+s(w)$ \ and \ $U_{\widetilde C_n}^{\widetilde B_{n+1}}(q)=\widetilde C_n(1,q,q)$

\smallskip

\item $\ell_{\widetilde D}(\phi(w))=\ell_{\widetilde C}(w)+r(w)+s(w)$ \ and \ $U_{\widetilde C_n}^{\widetilde D_{n+2}}(q)=\widetilde C_n(q,q,q)$

\smallskip

\item $\ell_{\widetilde{C}_{2n+1}}(\phi(w))=2\ell_{\widetilde C_{n}}(w)+s(w)$ \ and \ $U_{\widetilde C_n}^{\widetilde C_{2n+1}}(q)=\widetilde C_n(1,q,q^2)$

\smallskip

\item $\ell_{C_{2n}}(\phi(w))=2\ell_{\widetilde C_{n}}(w)-s(w)$ \ and \ $U_{\widetilde C_n}^{\widetilde C_{2n}}(q)=\widetilde C_n(1,q^{-1},q^2)$

\end{enumerate}
\end{proof}

\begin{remark}
If we set $a=0$ in Equation \eqref{PoincareSeriesAffineC}, then we recover the unfolding polynomials of finite type from Theorem \ref{thm:MainTheorem} parts (1)-(3).  In particular, we have 

$$U_{B_n}^{A_{2n-1}}(q)=\widetilde C_n(0,q,q^2),\quad U_{B_n}^{A_{2n}}(q)=\widetilde C_n(0,q^{-1},q^2),\quad U_{B_n}^{D_{n+1}}(q)=\widetilde C_n(0,q,q).$$
More generally, $$\widetilde C_n(0,t,q)=\sum_{w\in B_n} t^{s_B(w)}q^{\ell_B(w)}=(-tq;q)_{n}[n]_q!$$ recovers a well known formula for the $(t,q)$-distribution for the number of signs and inversions on signed permutations.
\end{remark}

\printbibliography

\end{document}